\documentclass[11pt, oneside]{article}    
\usepackage[margin=1.5in]{geometry}                   
\usepackage[utf8]{inputenc}
\usepackage[english]{babel}
\usepackage{amsmath,ulem}
\usepackage{amsthm}
\usepackage{graphicx}       
\usepackage{setspace}
\usepackage{dsfont}
\usepackage{mathtools}
\usepackage{graphicx}
\usepackage{amssymb}
\usepackage{enumerate}
\usepackage{epstopdf}
\usepackage{empheq}
\usepackage{dsfont}
\usepackage[utf8]{inputenc}
\usepackage{pgf,tikz}
\usepackage{mathrsfs}
\usepackage{todonotes}
\usepackage{authblk}
\usepackage{hyperref}
\usetikzlibrary{arrows}
\allowdisplaybreaks[4]

\newtheorem{prob}{Problem}
\newtheorem{theorem}{Theorem}

\newtheorem{lemma}{Lemma}
\newtheorem{proposition}[theorem]{Proposition}

\theoremstyle{remark}
\newtheorem{remark}{Remark}

  \newcommand\R{\mathbb{R}}
 \newcommand\N{\mathbb{N}}  
 \newcommand{\eps}{\varepsilon}
 \newcommand\cA{\mathcal{A}}
 
\newcommand\cE{\mathcal{E}}

\newcommand\Eell{\mathcal{E}^\ell}
\newcommand\Nell{N_{\ell}}

\newcommand\dist{{\rm dist}}
\newcommand\loc{{\rm loc}}
\newcommand\sgn{{\rm sgn}}
\newcommand\sinc{{\rm sinc }}
\newcommand\bvert{\bigg\vert}

\DeclareMathOperator{\sech}{sech}
\DeclareMathOperator{\supp}{supp}

\title{
\vspace{-1in}
Existence of solitary waves in one dimensional peridynamics} 
\author{Robert L. Pego\footnote{Email address:\url{rpego@cmu.edu}} }
\author{Truong-Son Van\footnote{Email address:\url{sonv@andrew.cmu.edu}}}
\affil{Department of Mathematical Sciences and
\\Center for Nonlinear Analysis \\ Carnegie Mellon University, Pittsburgh, PA 15213, USA}

\begin{document}
\date{August 31, 2018}
\nocite{*}
\maketitle
\begin{abstract}
		We give a rigorous proof of existence for solitary waves of a peridynamics model in one space dimension recently investigated by Silling ({\it J. Mech. Phys. Solids} 96:121--132, 2016).
We adapt the variational framework developed by Friesecke and Wattis ({\it Comm. Math Phys.} 161:391--418, 1994) for the Fermi-Pasta-Ulam-Tsingou lattice equations
to treat a truncated problem which cuts off short-range interactions, then pass to the limit.
\end{abstract}




\section{Introduction}
\subsection{Overview}
		Peridynamics is a relatively new nonlocal continuum model that was introduced in 2000 by Silling \cite{2000-S}.
A distinguishing feature of this model is that it consists of integro-differential equations that do not involve spatial derivatives.
For this reason, it has received considerable attention for its potential uses in modeling materials that have defects such as cracks.

Recently, Silling in \cite{S} investigated large-amplitude localized nonlinear waves---solitary waves---in a
peridynamics model in one space dimension. This model takes the form
		\begin{equation}\label{eq:motion1}
				u_{tt} = \int_{-\delta}^\delta f(u(x+\xi,t)-u(x,t),\xi) \,d\xi.
		\end{equation}
Here $x$ represents a material (Lagrangian) coordinate, $u$ represents the displacement field, $\xi$  is called a \textit{bond} and describes the relative position between two material points, $\delta$ is called the \textit{horizon} and represents the maximum reference distance between interacting material points. The function $f$ is a pairwise bond force density that is determined by the material model and is required to satisfy the anti-symmetry condition
		\begin{equation}\label{eq:f_rule}
			f(-\eta,-\xi)=-f(\eta,\xi).
		\end{equation}
Taking $f$ in a form that models an elastic material that hardens in compression (see Remark 2 below),
Silling demonstrated the presence of solitary waves in numerical simulations of
\eqref{eq:motion1} and studied their form by approximate analytical methods.

Our work, inspired by that of Friesecke and Wattis \cite{F-W}, aims to provide a general framework to rigorously prove the existence of solitary waves for \eqref{eq:motion1} by looking at the problem as a variational problem. In particular, we are looking for a solution in the form of a travelling wave,
		$$u(x,t) = q(x-ct),$$
		where $q:\R\to\R$ is monotone.
This provides a solution of \eqref{eq:motion1} if, for all real $z$,
		\begin{equation}\label{eq:limiting}
				c^2q''(z)= \int_0^\delta \Bigl(f(q(z+\xi)-q(z),\xi) - f(q(z)-q(z-\xi),\xi)\Bigr)\,d\xi.
		\end{equation}

		This equation is formally the Euler-Lagrange equation of the following general variational problem,
expressed in terms of a function $W$ (called the micropotential) which satisfies
\begin{equation}
\partial_{\eta} W(\eta,\xi) = f(\eta,\xi).
\end{equation}




		 \begin{prob}\label{generalization}
			Minimize
			\begin{equation}\label{func:T2}
					T(q):=\frac{1}{2} \int_{\R} q'(z)^2 dz
			\end{equation}
			subject to fixed potential energy
			\begin{equation}
					\cE(q) := \int_{\R}\int_0^\delta W\bigg(q(z+\xi)-q(z),\xi\bigg) \,d\xi \, dz = K
			\end{equation}
			where $K>0$.
		\end{prob}

		{\bf Assumptions on micropotential.} In this paper, we assume
that the micropotential takes the scaling form
\begin{equation}\label{Wform}
W(\eta,\xi) = V\bigg(\frac{\eta}{m(\xi)}\bigg)k(\xi)
\end{equation}
for some functions $m$ and $k$ that are chosen so that (\ref{eq:f_rule}) holds and $m(0)=k(0)=0$,
where $V$ is $C^2$ and the following hold:

	 \begin{enumerate}[(H1)]
		 \item\label{A1} $V$ is non-increasing (or decreasing), convex and superquadratic on  $(-\infty,0]$ (or $[0,\infty)$) with $V(0)=V'(0)=0$, $V''(0)>0$, and
\[
V''(x)  \le  V''(0) + c_1\vert x\vert^{\gamma_1} + c_2\vert x\vert^{\gamma_2},
\]
		 where $c_1,c_2 \ge1$ and $0< \gamma_1\le \gamma_2$.
		 \item\label{A2}  $k:\R \to \R$ is even, $m:\R\to\R$ is odd, both $m$ and $k$ vanish at only $0$ and are non-decreasing on $[0,\infty)$, and satisfy
			$$\int_0^\delta \frac{\xi^2 k(\xi)}{m(\xi)^{\gamma_2+2}} \, d\xi<\infty
\quad\mbox{and}\quad
\int_0^\delta \frac{k(\xi)}{m(\xi)} \, d\xi <\infty.$$
	 \end{enumerate}

	 Under these assumptions, we will prove the following theorem.
     We note that 
     under the long-wave approximation 
     $u(x+\xi,t)-2u(x,t)+u(x-\xi,t)\approx \xi^2 u_{xx}$,
     the linearization of equation \eqref{eq:motion1} takes the form
     $u_{tt} \approx c_0^2u_{xx}$, where 
     the ``speed of sound'' $c_0$ in this long-wave limit satisfies
     \begin{equation*}
         c_0^2 =  V''(0) \int_0^\delta \frac{ k(\xi) \xi^2}{m(\xi)^2} d\xi.
     \end{equation*}

	 \begin{theorem}\label{thm2}
		There exists a $K_1$ such that for $K>K_1$, there exists a constant $c> c_0$         and a function $q\in C^2(\R)$ so that $\cE(q)=K$ and $q$ solves equation \eqref{eq:limiting}.
		Furthermore, $q$ is increasing (decreasing) if $V$ is superquadratic on $[0,\infty)$ ($(-\infty,0]$).
		\end{theorem}
\hfill
		\begin{remark}
            Actually, while the function $q$ in Theorem \ref{thm2} provides a travelling wave solution to \eqref{eq:motion1}, we do not know whether it provides a minimizer for Problem \ref{generalization}. Rather, it is a solution to a minimization problem with a symmetrized micropotential, which is described in the next subsection.
		\end{remark}

\begin{remark}
    The speed of the travelling wave is actually always faster than both the group and phase velocities at any wave number $\kappa$.
Upon linearizing the equation around $u(\cdot +\xi, \cdot) - u(\cdot,\cdot)=0$ and taking solutions of the form $u(x,t) = e^{i\kappa x - i\omega t}$, we get the dispersion relation
    \begin{equation}
        \omega^2 = \int_0^\delta V''(0) 4\sin^2\bigg(\frac{\xi \kappa}{2}\bigg) \frac{ k(\xi)}{m^2} \,d\xi =  \kappa^2 \int_0^\delta 
        \sinc^2\bigg(\frac{\xi \kappa}{2}\bigg) 
        \,d\gamma(\xi),
    \end{equation}
    where $\sinc\, x = \frac{\sin x}{x}$ and 
    $d\gamma(\xi) = V''(0) \frac{k(\xi)\xi^2}{m(\xi)^2}\, d\xi$. 
    Since $\sinc^2 x\le1$ and $\int_0^\delta d\gamma(\xi)=c_0^2$,
    this tells us that the phase velocity $\omega/\kappa$ satisfies
    $(\omega/\kappa)^2<c_0^2$.
    Moreover, due to the identity
    \begin{equation*}
        \sinc^2(x) + \frac{x}{2}  \frac{d}{dx} \sinc^2(x) = \cos(x)\,\sinc(x),
    \end{equation*}
    we also compute that the group velocity (for $\omega$, $\kappa >0$) satisfies
    \begin{align*}
        \frac{d\, \omega}{d\, \kappa} &=  
          \frac{\kappa}{\omega} \int_0^\delta \cos \bigg(\frac{\xi\kappa}{2}\bigg) \sinc\bigg(\frac{\xi\kappa}{2}\bigg) \, d\gamma(\xi) < \frac{\kappa}{\omega}
          \sqrt{\frac{\omega^2}{\kappa^2}} \sqrt{\int_0^\delta d\gamma(\xi)}= c_0,
    \end{align*}
    by the Cauchy-Schwarz inequality.
    \end{remark}

\begin{remark}\label{r:Silling}
		The model that we study is a direct generalization of the one that was studied by Silling in \cite{S}. More specifically, the model investigated by Silling has the micro-potential
		\begin{equation}\label{Silling-micro}
		W(\eta,\xi) = \begin{cases}
				\frac{ \eta^2}{2\vert\xi\vert}(1-\frac{\eta}{3\xi}), \hspace{0.5in} \text{ if } \frac{\eta}\xi<0, \\
				\frac{ \eta^2}{2\vert\xi\vert} , \hspace{1in} \text{otherwise.}
		\end{cases}
\end{equation}
		This fits our framework with
		$$V(s) = \begin{cases}
				\frac12{ s^2}(1-\frac13{s}), \hspace{0.5in} \text{ if } s<0, \\
				\frac12{ s^2} , \hspace{1in} \text{otherwise,}
		\end{cases}$$
		$$m(\xi) = \xi, \qquad k(\xi) = \vert \xi\vert,$$
		$$\gamma_1=\gamma_2=1, \qquad c_1=c_2 = \frac{1}{2}.$$
\end{remark}

		We point out that while our general framework is good for showing the existence of travelling waves given large enough potential energy, Silling's model has a structure that allows us to prove the following result, showing that travelling waves exist even in the case of low potential energy.

		\begin{theorem}\label{thm:low-energy}
				There exists a monotone travelling wave solution to equation \eqref{eq:motion1} satisfying $\cE(q)=K$ for every $K>0$,
				where
				\begin{equation}\label{eq:interacting}
						f(\eta,\xi) = F(\eta/\xi) \sgn(\xi), \hspace{0.5cm} 0<\vert \xi\vert\le \delta,
				\end{equation}
				with 
				$$F(s) = \begin{cases}
								s - \frac{s^2}{2}, \hspace{0.7in} s< 0,\\
								s, \hspace{1in} \text{otherwise}.
				\end{cases}$$
		\end{theorem}

Concerning the asymptotic behavior of $q(z)$ for large $|z|$, we have only a little information.
In \cite{S}, by using Taylor's approximation, Silling derived an approximating ODE to \eqref{eq:limiting} and found an explicit solution to that ODE whose derivative has compact support. While this is numerically a good approximation, exact travelling waves for \eqref{eq:motion1} do {\it not} have compactly supported derivatives. This is proved in section 6 under a rather general assumption.

		Silling's model is the peridynamics counterpart of the discrete spring model that was studied by Fermi, Pasta, Ulam and Tsingou \cite{fermi1955studies}. In fact, Friesecke and Wattis \cite{F-W} rigorously showed that the discrete spring model possesses travelling wave solutions.
        While our result was inspired by that in \cite{F-W}, difficulties arose when we directly applied the method from \cite{F-W} due to the lack of control over the weak derivatives of functions in the Sobolev space $W^{1,2}_{\loc}(\R)$. Unlike the setting in \cite{F-W}, in which the authors only need to use the finiteness of kinetic energy to ensure the continuity of $q(\cdot + 1)-q(\cdot)$, we do not have the difference quotients $(q(\cdot+\xi)-q(\cdot) )/m(\xi)$ uniformly bounded as $\xi$ approaches zero. To be more specific, when we tried to apply the method from \cite{F-W} directly, we were not able to see how an analog of Lemma \ref{reformulate2}, a reformulation the nonvanishing condition of the minimizing sequence, would hold. 
        
        To overcome this problem, we find the existence of travelling waves in an approximate problem obtained by cutting off short bonds, and then prove that the approximate solutions converge to a solution of the main problem. The ability to prove such convergence relies heavily on an improved potential energy estimate (Lemma \ref{lem:energy_revisit}) and the monotonicity of the solutions of the approximate problem.

        We also note that the truncation near zero is not related to the assumptions on $m$ and $k$. It is mainly to deal with the lack of uniform boundedness of the difference quotients.

It is plausible that the existence of solitary waves can be obtained with 
different or more general structural assumptions on the micropotential from those we impose here.
We have chosen to treat micropotentials in the scaling form of \eqref{Wform} because they can conveniently represent 
a variety of typical peridynamic force densities,
such as arise, e.g., by finite-horizon truncation of energies involving fractional derivatives.

		\subsection{Symmetrization}\label{solitary-waves}

		In order to prove existence of solitary waves, we will symmetrize our potential $W$ by replacing $V$ with a function $\overline{V}$ that is even and superquadratic on $\R$. More specifically, let $I$ be the half line on which $V$ is superquadratic (i.e. $(-\infty,0]$ or $[0,\infty)$). Define
		$$\overline{V}(x) := \begin{cases}
				V(x), \qquad x\in I \\
				V(-x), \qquad -x \in I
		\end{cases}. $$

		Define then $$\overline{W}(\eta,\xi) : = \overline{V}(\frac{\eta}{m(\xi)})k(\xi).$$

		It turns out that problem \ref{generalization} with the potential $\overline{W}$ will have a minimizer that is monotone. Furthermore, if $q$ is a minimizer of this problem, then $-q$ is also a minimizer (due to the symmetric nature of $\overline{V}$).

		Once we prove the existence of minimizers for this symmetrized problem, we will see that if the original $V$ is superquadratic on $[0,\infty)$ ($(-\infty,0]$) then the increasing (decreasing) minimizer of the symmetrized problem will be a solitary wave to the original problem \ref{generalization}.

		We note that the solitary wave found here may not be the minimizer of the original variational problem 1.

        \begin{remark}
            It is not necessary to symmetrize the potential if $V$ is already strictly convex. The only place that this is utilized heavily is Lemma \ref{lem:monotone}, where we prove monotonicity of minimizers by exploiting the strict convexity of $\overline{V}$. This is not necessary to prove existence of the minimizer, where we only need to exploit the one-sided superquadraticity of $V$.
        \end{remark}

		\subsection{Truncation} \label{sub:existence}

Next, we introduce the following truncated problem:
\\
		\begin{prob}\label{approximate}
		 Minimize
		 \begin{equation*}
				 T(q):=\frac{1}{2} \int_{\R} q'(z)^2 dz
		 \end{equation*}
		 subject to a fixed potential energy
         \begin{equation}\label{potential_energy}
				 \Eell(q) := \int_{\R}\int_\ell^\delta \overline{W}\bigg(q(z+\xi)-q(z),\xi\bigg) d\xi dz = K
		 \end{equation}
		 where $\ell\in(0,\delta)$, $K>0$ and $\overline{W}$ is defined above.
	 \end{prob}

		It turns out that we can solve this problem by adapting the technique of \cite{F-W}:
\\
				\begin{theorem}\label{thm1}
						There exists a $K_0$ such that for all $K>K_0$, there exists  $q^\ell\in C^2(\R)$ so that $\Eell(q^\ell)=K$ and $q^\ell$ solves Problem \ref{approximate}. Furthermore, it solves the Euler-Lagrange equation
						\begin{equation}\label{eq:travelling-wave-cutoff}
								c^2(q^\ell)''(x)   = \int_{\ell}^{\delta} \big[f(q^\ell(x+\xi)-q^\ell(x),\xi)-f(q^\ell(x)-q^\ell(x-\xi),\xi)\big]d\xi
								\end{equation}
								where $c^2>0$ is the inverse of the Lagrange multiplier.
				\end{theorem}

		We will then extract a limit
		$$q^\ell \to q$$
		along some subsequence and show that $q$ is non-trivial and solves \eqref{eq:limiting}.

        \subsection{Strategy and plan of the paper}\label{strategy}
				To summarize, the strategy to establish the existence of a solution to (\ref{eq:limiting}) is the following:

						\begin{enumerate}
								\item Prove the existence of a minimizer $q^\ell$ to Problem \ref{approximate},
								\item Show that $q^\ell$ is monotone,
                                \item Show that as $\ell\to 0$, a subsequence of the $q^\ell$ converges to a minimizer of Problem \ref{generalization} with the symmetrized potential $\overline{W}$.
								\item Conclude that this minimizer of the problem associated with $\overline{W}$ is a function that satisfies (\ref{eq:limiting}).
						\end{enumerate}

				The existence of minimizers to the truncated Problem \ref{approximate} is proven in section \ref{existence}. Section \ref{properties} derives various properties for the minimizers and completes the proof of Theorem~\ref{thm1}.
The analysis in sections~\ref{existence} and \ref{properties} is similar to that in \cite{F-W}, so
readers who are familiar with that can skip the details in these sections without missing any major concept.
The existence of travelling waves in the original problem will be proven in Section \ref{limiting}. Theorem~\ref{thm:low-energy} is shown in Section \ref{low-energy}. Finally, we discuss compactness of the travelling waves' derivatives in Section \ref{discussion}.

                

\section{Existence of a minimizer to the truncated problem} \label{existence}
We establish Theorem~\ref{thm1} in this section, which deals with the truncated Problem~\ref{approximate}. 
Since we only deal with the symmetrized potential, we will write $\overline{V}$ as $V$ in this section, unless specifically stated otherwise.
We also note that our proof follows almost exactly as in \cite{F-W} with some modifications needed due to the fact that our potential is an integral, not a function as in \cite{F-W}. 

\subsection{Notations}


		We will be working on the following Hilbert space
		$$H:=\bigg\{q\in W^{1,2}_\loc(\R): \Vert q'\Vert_{L^2(\R)}<\infty, q(0)=0\bigg\} $$
		where the inner product is given by $\langle q,p\rangle = \int_{\R} q'p'$ and $\Vert q\Vert = \Vert q' \Vert_{L^2(\R)}$.
        For convenience, we define 
        \begin{equation}\label{kinetic_energy_K}
            T^\ell_K:= \inf_{\cA^\ell_K} T, \qquad   \cA^\ell_K:= \{q\in H: U(q(\cdot)) \in L^1(\R), \Eell(q)=K\}.
    \end{equation}
    The analysis in this section and the next is performed for each fixed $\ell \in (0,\delta)$. Therefore, for convenience we suppress the explicit dependence on $\ell$ frequently 
in these sections to make the notations less cluttered. 
                The results in sections~\ref{limiting} and \ref{low-energy} require close attention to different values of $\ell$ so we will explicitly write the $\ell$ dependence there in all of the calculations. 

It is also convenient to define the piecewise linear function $q_{\Lambda,L}$ by
		$$q_{\Lambda, L}(z)=\begin{cases}
				0, & z\le 0,\\
				\Lambda z, &z\in [0,L],\\
				\Lambda L, & z>L.
		\end{cases}$$
%
%
%
%
%

\subsection{Analysis}

Let us start with the concentration-compactness lemma.

\begin{lemma}[concentration-compactness]\label{concentration-compactness}
	Let $\{q^k\}$ be a sequence in $W^{1,2}_{loc}(\R)$ such that there exists a $C>0$ so
	$$\sup_k \Vert (q^k)'\Vert_{L^2(\R)}\le C$$ and that
	\begin{equation}
		\int\limits_\R U^k \equiv K
	\end{equation}
	where $K>0$ and
	\begin{equation} \label{d:Uk}
		U^k(z)=U(q^k;z):= \int_{\ell}^\delta V\big(\frac{q^k(z+\xi)-q^k(z)}{m(\xi)}\big)k(\xi) \, d\xi.
	\end{equation}
	Then, up to a subsequence, $q^k$ satisfies exactly one of the following
	\begin{enumerate}[i.]
		\item (compactness) There exists $y_k \in \R$ such that $U^k(\cdot + y_k)$ is tight, i.e., $\forall \eps>0$, $\exists R<\infty$ such that
		$$\int\limits_{\R\backslash B_R(0)} U^k(\cdot + y_k)\le \eps$$ for all $k$.

		\item (vanishing) $$\lim_{k\to\infty}\sup_{y\in\R} \int\limits_{B_R(y)} U^k(\cdot) = 0$$ for all $R<\infty$.

		\item (splitting) There exists $\alpha\in (0,K)$ such that $\forall \eps>0$, $\exists k_0$
		such that $\forall k\ge k_0$,
		$\exists\,q_1^k, q_2^k\in W^{1,2}_{loc}(\R)$, $\Vert (q_i^k)'\Vert_{L^2(\R)}<\infty$
		and the following is true:
			 $$\Vert U^k - (U^k_1+U^k_2)\Vert_{L^1(\R)} \le \eps,$$
		$$\bigg\vert \int\limits_\R U_1^k -\alpha\bigg\vert \le \eps,$$
		$$\bigg\vert \int\limits_\R U_2^k -(K-\alpha)\bigg\vert \le \eps,$$
		and
		$$\frac{1}{2}\int\limits_\R ((q^k)')^2 -\frac{1}{2}\int\limits_\R ((q^k_1)')^2 -\frac{1}{2}\int\limits_\R ((q^k_2)')^2 \ge 0$$
		with $$\dist(\supp((q_1^k)'),\supp((q_2^k)'))\to \infty$$ as $k\to\infty$.
		Here $U_j^k(x):= U(q_j^k;x)$ where $j=1,2$.
	\end{enumerate}
\end{lemma}
\begin{proof}

		{\bf Step 1.} For each $k$, define a concentration function $Q_k:\R^+\to \R^+$ such that
		$$0\le Q_k(R):= \sup_{y\in \R} \int\limits_{B_R(y)} U^k \le K$$ for all $k$.

		Since $Q_k$'s are increasing functions that are uniformly bounded, up to a subsequence, we have that $Q_k$ converges to some  non-negative increasing function $Q$, pointwise. Define $$\alpha := \lim_{R\to\infty}Q(R)\in [0,K].$$ We have 3 cases:
		\begin{itemize}
			\item $\alpha = 0$. This implies (ii).
			\item $\alpha =K$. This implies (i).
			\item $\alpha\in (0,K)$. We need to prove that this implies (iii).
		\end{itemize}
		To show the last item, we proceed as following.

		{\bf Step 2.}  Fix $\eps>0$, since $Q\nearrow \alpha$, there exists $R$ such that $Q(R-\delta)>\alpha-\eps$. Since $Q_k \to Q$ pointwisely and $Q_k$'s are increasing function, for large enough $k$, pick $y_k$ so that

		$$\int_{B_R(y_k)} U^k > \alpha-\eps.$$

		Furthermore, since $\lim_{R\to\infty} Q(R)=\alpha$, we can find $R^k\to\infty$ such that
		$$Q_k(R^k+\delta) < \alpha+\eps.$$
		Let $R_1^k \in [R,R + \frac{1}{3}(R^k-R)]$ and $R_2^k\in [R+\frac{2}{3}(R^k-R),R^k]$ where $R_1^k<R_2^k$ to be specified later.

		Define continuous functions $q_1^k$ and $q_2^k$ so that
		\begin{equation*}
				(q_1^k)'=\begin{cases}
						(q^k)' , & z\in B_{R_1^k}(y_k),\\
						0, & \text{ otherwise },
				\end{cases}
		\end{equation*}
		and
		\begin{equation*}
				(q_2^k)'=\begin{cases}
						0 , & z\in B_{R_2^k}(y_k),\\
						(q^k)', & \text{ otherwise }.
				\end{cases}
		\end{equation*}
		More specifically,
		\begin{equation}
			q_1^k(z) := \begin{cases}
									q^k(y_k-R_1^k),  &z\le y_k -R_1^k,\\
									q^k(z),  & z\in B_{R_1^k}(y_k),\\
									q^k(y_k+R^k_1),  &z\ge y_k+R_1^k.
							\end{cases}
		\end{equation}

		\begin{equation}
			q_2^k(z) := \begin{cases}
									q^k(z)-q^k(y_k-R_2^k),  & z\le y_k -R_2^k,\\
									0,  & z\in B_{R_2^k}(y_k),\\
									q^k(z) - q^k(y_k + R_2^k), &z\ge y_k+R_2^k.
							\end{cases}
		\end{equation}

		By direct computation, we have that
		$$\frac{1}{2}\int\limits_\R ((q^k)')^2 -\frac{1}{2}\int\limits_\R ((q^k_1)')^2 -\frac{1}{2}\int\limits_\R ((q^k_2)')^2 =\frac{1}{2}\int\limits_{B_{R_2}(y_k)\backslash B_{R_1}(y_k)} ((q^k)')^2\ge 0$$ and $$\dist(\supp((q_1^k)'),\supp((q_2^k)'))\to \infty$$ as $k\to\infty$.

		We now need to choose good $R_1^k$ and $R_2^k$.

		{\bf Step 3.} We next get some estimate for $U_j^k$, for $j=1,2$. First, from Cauchy-Schwarz inequality and since $\ell>0$,
		\begin{align*}
				U^k_j(z) &:= U(q^k_j;z) =\int\limits_{\ell}^\delta V\bigg(\frac{q_j^k(z+\xi) - q_j^k(z)}{m(\xi)}\bigg)k(\xi) \, d\xi\\
						& \le \int_{\ell}^{\delta}V\bigg(\frac{\xi^{1/2}(\int_0^\xi \vert (q^k_j)'(z+s)\vert^2 ds)^{1/2}}{m(\xi)} \bigg)k(\xi)\,  d\xi \\
						& \le C_\ell \int_\ell^\delta \bigg(\int_0^\delta \vert (q^k_j)'(z+s)\vert^2 ds\bigg)^{1/2} \frac{\xi^{1/2}k(\xi)}{m(\xi)}\,  d\xi \\
					&\le M_\ell\bigg(\int_0^\delta \vert (q^k)'(z+s)\vert^2ds \bigg)^{1/2}
			\end{align*}
            where $C_\ell$ is the local Lipschitz constant of $V$ depending on $C$ in the hypothesis and $\ell$. (We remark that this is the key new estimate that makes the rest of the proof work almost exactly as in \cite{F-W} again.)

		{\bf Step 4.}  Let
		$$D_{i,k}^{-} := (y_k - R_i^k - \delta, y_k- R_i^k],$$
		$$D_{i,k}^{+} := [y_k - R_i^k , y_k- R_i^k+ \delta),$$
		$$E_{i,k}^{-} := (y_k + R_i^k - \delta,y_k + R_i^k], $$
		$$E_{i,k}^{+} := [y_k + R_i^k ,y_k + R_i^k + \delta). $$

		\begin{figure}[h!]
		\begin{tikzpicture}
				\draw[latex-latex] (-7,0) -- (7,0) ; 
				\draw[shift={(0,0)},color=black] (0pt,0pt) -- (0pt,-3pt) node[below] {$y_k$};
				\draw[shift={(-5,0)},color=black]  node[below] {$y_k-R_2^k$};
				\draw[shift={(-2,0)},color=black]  node[below] {$y_k-R_1^k$};
				\draw[shift={(5,0)},color=black]  node[below] {$y_k+R_2^k$};
				\draw[shift={(2,0)},color=black]  node[below] {$y_k+R_1^k$};

				\draw[shift={(1.5,0)},color=black]  node[above] {$E^{-}_{1,k}$};
				\draw[shift={(4.5,0)},color=black]  node[above] {$E^{-}_{2,k}$};
				\draw[shift={(5.5,0)},color=black]  node[above] {$E^{+}_{2,k}$};
				\draw[shift={(-1.5,0)},color=black]  node[above] {$D^{+}_{1,k}$};
				\draw[shift={(-2.5,0)},color=black]  node[above] {$D^{-}_{1,k}$};
				\draw[shift={(-5.5,0)},color=black]  node[above] {$D^{-}_{2,k}$};

				\draw[*-o] (-4.9,0) -- (-6.0,0);
				\draw[very thick] (-5,0) -- (-5.9,0);
				\draw[o-] (-1.0,0) -- (-2.00,0);
				\draw[*-o] (-1.9,0) -- (-3.0,0);
				\draw[very thick] (-1.1,0) -- (-2.9,0);
				\draw[o-*] (1.0,0) -- (2.1,0);

				\draw[very thick] (1.1,0) -- (2,0);
				\draw[o-] (4.0,0) -- (5.00,0);
				\draw[*-o] (4.9,0) -- (6.0,0);
				\draw[very thick] (4.1,0) -- (5.9,0);
		\end{tikzpicture}
		\caption{Intervals}
		\end{figure}
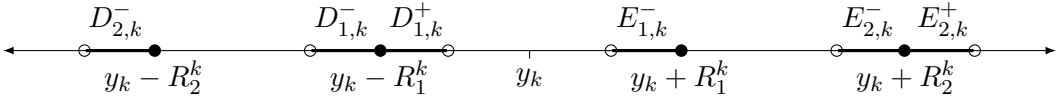

		We have that, $z<y_k - R_1^k$ implies $(q_1^k)'=0$. So,
		 \begin{align*}
			\int_{D_{1,k}^{-}} U_1^k(z) dz  &\le M_\ell\int_{D_{1,k}^{-}}\bigg(\int\limits_{z}^{z+\delta} \vert (q^k_1)'(s)\vert^2ds\bigg)^{1/2}  dz\\
				&\le M_\ell\int_{D_{1,k}^{-}}\bigg(\int_{D_{1,k}^{+}} \vert (q^k_1)'(s)\vert^2ds\bigg)^{1/2}  dz\\
				&= M_\ell\delta\bigg(\int_{D_{1,k}^{+}} \vert (q^k_1)'(s)\vert^2ds\bigg)^{1/2}
		\end{align*}
		where the constant $M_\ell$ may differ from line to line.
		Similarly, we have that

		 \begin{align*}
			\int_{E_{1,k}^{-}} U_1^k &\le M_\ell \bigg(\int_{E_{1,k}^{-}}((q^k)')^2\bigg)^{1/2}, \\
			\int_{D_{2,k}^{-}} U_2^k &\le M_\ell \bigg(\int_{D_{2,k}^{-}}((q^k)')^2\bigg)^{1/2}, \\
			\int_{E_{2,k}^{-}} U_2^k &\le M_\ell \bigg(\int_{E_{2,k}^{+}}((q^k)')^2\bigg)^{1/2},
	\end{align*}

		where we can take $M_\ell$ to be the same for all estimates.

        {\bf Step 5.} Now, we partition $[R+\delta, R +\frac{1}{3}(R^k-R)]$ into $\frac{R^k-R}{3\delta} -1$ intervals and use the fact that  $\int\limits_{\R} ((q^k)')^2 \le C^2$ to get the following.
			\begin{align*}
			 \bigg(\frac{(R^k-R)}{3\delta}-1\bigg)&\min_{R^k_1 \in [R+\delta,R+\frac{1}{3}(R^k-R)]}\int\limits_{ E_{1,k}^{-}} U^k_1 \le \int_{[R,R+\frac{1}{3}(R^k-R)} U_1^k\\
			 &\le M_\ell\bigg(\int\limits_{[R,R+\frac{1}{3}(R^k-R)+\delta]}((q^k)')^2\bigg)^{1/2} \le M_\ell C.
	\end{align*}

	Thus,
	\begin{equation}
			\min_{R_1^k \in [R+\delta,R+\frac{1}{3}(R^k-R)]}\int\limits_{E_{1,k}^{-}} U_1^k\le \tilde{\delta}(k)
	\end{equation}
	where $$\tilde{\delta}(k):= \frac{M_\ell C}{\frac{R^k-R}{3\delta}-1}\xrightarrow{k\to\infty} 0.$$

	Similarly, the above works if we replace $E^{-}_{1,k}$ by $D^{-}_{1,k}$ and so
	\begin{equation}
		\min_{R_1^k \in [R+\delta,R+\frac{1}{3}(R^k-R)]}\int\limits_{E_{1,k}^{-} \bigcup D_{1,k}^{-}} U_1^k\le 2\tilde{\delta}(k).
\end{equation}

We also have the following
	\begin{equation}
			\begin{aligned}
			\min_{R_2^k \in [R+\frac{2}{3}(R^k-R),R^k]}\int\limits_{D_{2,k}^{-} \bigcup E_{2,k}^{-}} U_2^k &\le 2\tilde{\delta}(k).
		\end{aligned}
		\end{equation}

		{\bf Step 6.} Choose $R_1^k$ and $R_2^k$ so that the minima above are obtained. We then have that

			\begin{align*}
			\int_\R \vert U^k - U_1^k - U_2^k\vert &= \int\limits_{D_{1,k}^{-}\bigcup E_{1,k}^{-}}\vert U^k-U_1^k\vert
									\qquad + \int\limits_{D_{2,k}^{-}\bigcup E_{2,k}^{-}}\vert U^k-U_2^k\vert \\
									& \qquad+ \bigg(\int\limits_{y_k-R_2^k}^{y_k-R_1^k-\delta}+\int\limits_{y_k+R_1^k}^{y_k+R_2^k-\delta}\bigg) U^k \\
									&\le \int\limits_{D_{1,k}^{-}\bigcup E_{1,k}^{-}}  (U_1^k + U^k)
										\qquad+  \int\limits_{D_{2,k}^{-}\bigcup E_{2,k}^{-}} U_2^k + U^k\\
										&\qquad + \bigg( \int\limits_{y_k-R_2^k}^{y_k-R_1^k-\delta} +\int\limits_{y_k+R_1^k}^{y_k+R_2^k-\delta}\bigg) U^k\\
									&\le 4\tilde{\delta}(k) + \int_{B_{R^k+\delta}(y_k)}U^k -\int_{B_{R-\delta}(y_k)} U^k\\
									&< 4\tilde{\delta}(k) + (\alpha + \eps)  - (\alpha-\eps) = 4\tilde{\delta}(k) + 2\eps.
			\end{align*}

		Furthermore,
		\begin{equation}
				\begin{aligned}
			\int_\R U_1^k &= \int\limits_{D_{1,k}^{-} \bigcup E_{1,k}^{-}} U_1^k \qquad + \int\limits_{y_k-R_1^k}^{y_k+R_1^k-\delta}U^k_1
			\end{aligned}
		\end{equation}

		But we know that, from the definitions of $R, R^k$ and $R_1^k$, for large enough $k$ so that $R^k_1-\delta \ge R$,
		$$ \int\limits_{y_k-R_1^k}^{y_k+R_1^k-\delta}U^k \in (\alpha-\eps,\alpha+\eps).$$
		So,
		$$\bigg\vert \int_\R U_1^k dz - \alpha\bigg\vert \le \eps.$$
		Similarly, we have, for large enough $k$,
		\begin{equation}
			\begin{aligned}
			\int_\R U_2^k &= \int\limits_{D_{2,k}^{-}\bigcup E_{2,k}^{-}}  U_1^k \qquad + \bigg( \int\limits_{-\infty}^{y_k-R_2^k-\delta} +\int\limits_{y_k+R_2^k}^\infty \bigg) U^k_2
			\end{aligned}
		\end{equation}
		and
		$$\bigg( \int\limits_{-\infty}^{y_k-R_2^k-\delta} +\int\limits_{y_k+R_2^k}^\infty \bigg) U^k \in ((K-\alpha)-\eps,(K+\alpha)-\eps).$$
		So,
		$$\bigg\vert \int_\R U_2^k dz - (K-\alpha)\bigg\vert \le \eps.$$
		This proves the last item in step 1 and hence the lemma.
\end{proof}


\begin{remark}\label{technical}
	For each minimizing sequence $\{q^k\}\subseteq \cA_K$ of $T$, by replacing $q^k$ by
	$$\tilde{q}^k =\begin{cases}
        q^{k}(-a), & z\le -a \\
		q^k(z), & -a\le z\le a\\
		q^k(a) + q_{\Lambda,1}(z-a-1), & z\ge a
    \end{cases} $$ for sufficiently large $a$, and $\Lambda$ chosen to ensure that $\cE^\ell(\tilde{q}^k) = K$, we can assume each function of the minimizing sequence has a compact-support derivative.
\end{remark}
\hfill

\begin{lemma}\label{mon}
		The map $K\mapsto T_K, K\in [0,\infty)$ is monotone increasing and continuous.
\end{lemma}
\begin{proof} The argument of this lemma is basically the same as in Lemma 1d in \cite{F-W}.

		The monotone part is based on a simple scaling argument.
		Let $\alpha\le K$, $q\in \cA_K$. Because $\Eell(\Lambda q)=0$ when $\Lambda=0$ and $\Eell(\Lambda q)=K$ when $\Lambda=1$, there exists $\Lambda_0\in [0,1]$ so that $\Eell(\Lambda_0 q)=\alpha$. We have that
		$$T_\alpha \le T(\Lambda_0 q) = \Lambda_0^2 T(q) \le T(q).$$
		This works for all $q\in \cA_K$. Thus, $T_\alpha\le T_K.$

				To see $K\mapsto T_K$ is continuous, we proceed as follows. Since $T_K$ is monotone in $K$, it suffices to show that there exists  $\eta(r)\to 0$ as $r\to 0^+$ so that
				for all $K$,
				$$T_{K+r}-T_K \le \eta(r).$$

		Fix $K$ and $r>0$. Let $\eps>0$, $q_K \in \cA_K$ such that $T(q_K)\le T_K+\eps$. Consider the function
		$$q(z):= \begin{cases}
				q_K(z), &z\le a\\
				q_K(a)+ q_{\Lambda,1}(z-a), & z> a
		\end{cases}, $$ where $\Lambda$ and $a$ will be specified later.

		Observe that $q\in H$ and
		$$\Eell(q) = \int_{-\infty}^{a-\delta} U(q_K;\cdot) +\int_{a-\delta}^a U(q;\cdot) + \Eell(q_{\Lambda,1}).$$

		Also,
		$$T(q)- T_K \le T(q)- T(q_K)+\eps \le T(q_{\Lambda,1})+\eps = \frac12\Lambda^2+\eps.$$

		 Now, define
		$$\Lambda_1(r):=\inf\{ \Lambda>0: \Eell(q_{\Lambda,1})=2r\}. $$

		Since $\Eell(q_{\Lambda,1})$ is increasing and continuous in $\Lambda$ and $\Eell(q_{0,1})=0$, pick $0<\Lambda \le \Lambda_1(r)$ so that $r<\Eell(q_{\Lambda,1})<K+r$. Then, we have that
		$$\Eell(q)\to K + \Eell(q_{\Lambda,1})>K+r$$ as $a\to\infty$ and
		$$\Eell(q)\to \Eell(q_{\Lambda,1})< K+r$$ as $a\to-\infty$. By continuity of the integral, there exists $a\in \R$ such that
		$$\Eell(q) = K+r.$$

		Thus,
		$$T_{K+r} - T_K \le T(q)- T_K \le T(q_{\Lambda,1})+\eps \le \frac{1}{2} \Lambda_1(r)^2 +\eps.$$
		Since $\eps$ is arbitrary,
		$$T_{K+r} - T_K \le \frac{1}{2}\Lambda_1^2(r).$$
		Observe that $\Lambda_1(r)\to 0$ as $r\to0$ independently of $K$. Thus, we define $$\eta(r):=\frac{1}{2}\Lambda_1^2(r)$$ and the result follows.
\end{proof}


\begin{lemma}\label{reformulate1} Let $U$ be as above and $K>0$ be fixed. Then
the following are equivalent:
		\begin{enumerate}[i.]
			\item No minimizing sequence $\{ q^k\}\subseteq \cA_K$ splits, i.e., satisfies Lemma \ref{concentration-compactness}.iii.
			\item $T$ satisfies a subadditivity condition
			\begin{equation}
					T_K < T_\alpha + T_{K-\alpha}, \hspace{0.3in} \forall \alpha\in (0,K). \label{subadditivity} \tag{S}
				\end{equation}
		\end{enumerate}
\end{lemma}
\begin{proof}
		The proof of this goes almost exactly the same as in \cite{F-W}.

		($(ii)\implies (i)$). Suppose, by contradiction, that there exists a minimizing sequence $\{q^k\}\subseteq \cA_K$ that splits for some $\alpha\in (0,K)$. Define $$\alpha_k:= \Eell(q^k_1),$$ and $$\beta_k:= \Eell(q^k_2).$$
		As $k\to\infty$, by continuity of $T_K$, we have that
		$$T_K \ge \liminf_{k\to\infty}(T(q_1^k)+T(q_2^k)) \ge \liminf_{k\to\infty}(T_{\alpha_k} + T_{\beta_k}) = T_\alpha + T_{K-\alpha},$$ quod est absurdum.

		($(i)\implies (ii)$). Suppose, by contradiction, that (\ref{subadditivity}) does not hold, i.e., $\exists \alpha \in (0,K)$ such that $$T_K\ge T_\alpha + T_{K-\alpha}.$$
		We want to construct a minimizing sequence that splits. Let $\{q_\alpha^k\} \subseteq \cA_{\alpha}$ and $\{q_{K-\alpha}^k\} \subseteq \cA_{K-\alpha}$ be minimizing sequence of $T$ under the respective constraints. By remark \ref{technical}, we can assume that the supports of $(q^k_\alpha)'$ and $(q^k_{K-\alpha})'$ are contained in $B_{R_k}(0)$ for some $R_k>0$. Then the sequence $$q^k(z):= q^k_{K-\alpha}(z+R_k+k) + q^k_{K-\alpha}(z-R_k-k)+C_k$$ where $C_k$ is chosen to make $q^k(0)=0$ works as desired.
\end{proof}

Next, we introduce the quantity
$$[q(x)]_{\ell}=\sup_{\ell\le\xi\le\delta}\bigg\vert \frac{q(x+\xi)-q(x)}{m(\xi)}\bigg\vert,$$
and note that this bounds the argument of $V$ in the expression for $U(q;z)$ defined in
\eqref{d:Uk}. As in Step 3 of the proof of Lemma~\ref{concentration-compactness},
by Cauchy-Schwarz we find the bound
\[
 \Vert [q]_\ell \Vert_{\infty} \le \frac{\sqrt{2\delta T(q)}}{m(\ell)}.
 \]

\begin{lemma} [Uniform modulus of continuity of $U$] \label{unif_cont}
		Let $M>0$. There exists a constant $C_1$ such that for all $q\in W^{1,2}_\loc(\R)$ with $T(q)\le M^2$,
		$$\vert U(q;z_1)-U(q;z_2)\vert \le C_1\vert z_1-z_2\vert^{1/2}$$
		for any $z_1,z_2\in \R$.
\end{lemma}
\begin{proof}
		Let $z_1, z_2\in \R$, writing $\eta(z,\xi)=q(z+\xi)-q(z)$ we have that
				\begin{align*}
						\vert U(q;z_1) - U(q;z_2)\vert
								&= \bvert \int_{\ell}^\delta \bigg[V\bigg(\frac{q(z_1+\xi)-q(z_1)}{m(\xi)}\bigg) - V\bigg(\frac{q(z_2+\xi)-q(z_2)}{m(\xi)}\bigg) \bigg]k(\xi) \, d\xi\bvert\\
								& \le \int_{\ell}^\delta \bigg\vert \int_{\eta(z_2,\xi)/m(\xi)}^{\eta(z_1,\xi)/m(\xi)} V'(s)ds\bigg\vert k(\xi) \, d\xi\\
								&\le \int_{\ell}^\delta \tilde{C} \bigg\vert \frac{\eta(z_1,\xi)}{m(\xi)} -\frac{\eta(z_2,\xi)}{m(\xi)}\bigg\vert k(\xi) \, d\xi\\
								&\le \tilde{C}'\int_{\ell}^\delta \bigg\vert \int_{z_1}^{z_2} [q'(s+\xi)-q'(s)]ds\bigg\vert \, d\xi\\
								& \le \tilde{C}'\int_{\ell}^\delta \vert z_1-z_2\vert^{1/2}\Vert q'\Vert_{L^2} \, d\xi\\
								& = (\delta-\ell)\tilde{C}'M\vert z_1-z_2\vert^{1/2} = C_1\vert z_1-z_2\vert^{1/2}
				\end{align*}

		where   $\tilde{C}= V'(\Vert[q]_{\ell}\Vert_\infty)$.
\end{proof}

Now, for convenience, we introduce the following notation:

$$\Nell = \int_\ell^\delta \frac{\xi^2 k(\xi)}{m(\xi)^2} \, d\xi . $$
\begin{lemma}\label{ineq:liminf}
		Define
		$$T_{K,\eps} = \inf\{ T(q): q\in \cA_K, \Vert[q]_{\ell}\Vert_\infty\le \eps \}.$$
				Then
				\begin{equation}\label{e:Tke-eq}
				\liminf_{\eps\to 0}T_{K,\eps} =\frac{K}{\Nell V''(0)}.
				\end{equation}
\end{lemma}
\begin{proof}
		First, we prove the following inequality
		\begin{equation}\label{ineq:lim}
			\liminf_{\eps\to 0}T_{K,\eps} \ge\frac{K}{\Nell V''(0)}.
		\end{equation}
		Let $\{q^k\} \subseteq \cA_K$ be a sequence such that $\Vert[q^k]_{\ell}\Vert_\infty\to 0$. Let $R_{\ell}^k:=[-\Vert[q^k]_{\ell}\Vert_\infty, \Vert[q^k]_{\ell}\Vert_\infty]$. Then, let $M_\ell=\sup_{x\in R_{\ell}^k}\vert V''(x)\vert$,
		\begin{equation*}
				\begin{aligned}
						U(q^k;z) &= \int_{\ell}^\delta  V\bigg(\frac{q^k(z+\xi)-q^k(z)}{m(\xi)}\bigg) k(\xi) \, d\xi\\
								&\le \frac{1}{2} M_{\ell} \int_{\ell}^\delta \bigg(\frac{q^k(z+\xi)-q^k(z)}{m(\xi)} \bigg)^2 k(\xi) \, d\xi \\
                                &\le \frac{1}{2} M_{\ell} \int_{\ell}^\delta \bigg(\int_{z}^{z+\xi} \vert {q^k}'(s)\vert^2 ds \bigg) \frac{\xi k(\xi)}{m(\xi)^2} \, d\xi.
				\end{aligned}
		\end{equation*}
		Thus,
		\begin{equation*}
				\begin{aligned}
                    \int_{\R} U(q^k;z)dz &\le \frac{1}{2} M_\ell \int_{\R} \int_{\ell}^\delta \bigg( \int_{z}^{z+\xi} \vert {q^k}'(s)\vert^2 ds \bigg) \frac{\xi k(\xi)}{m(\xi)^2} \, d\xi \, dz \\
                    &= \frac{1}{2} M_\ell \int_{\ell}^\delta \int_{\R} \bigg(\int_{z}^{z+\xi}\vert {q^k}'(s)\vert^2 \, ds  \bigg)\frac{\xi k(\xi)}{m(\xi)^2}\, dz \, d\xi \\
                    &= \frac{1}{2}M_\ell \int_{\ell}^\delta \frac{\xi^2 k(\xi)}{m(\xi)^2}\int_{\R} \vert {q^k}'(s)\vert^2 \, ds d\xi\\
						&= M_\ell \Nell T(q^k).
				\end{aligned}
		\end{equation*}
		Therefore, by assumption, $$\liminf_{k\to\infty} T(q^k) \ge \frac{K}{\Nell V''(0)}.$$
		Since $\{q^k\}$ is arbitrary, the inequality \eqref{ineq:lim}  follows.

		Next, we prove the equality \eqref{e:Tke-eq} by employing the piecewise linear function $q_{\Lambda, L}$. We note that,
		$$T(q_{\Lambda,L}) = \frac{1}{2} \Lambda^2 L$$
		and

		\begin{equation} \label{eq:qlambdal}
				\int_{\R} U(q_{\Lambda,L})=\int_{\ell}^\delta \bigg[(L-\xi)V(\frac{\Lambda\xi}{m(\xi)})k(\xi) + g(\xi)  \bigg]d\xi 
		\end{equation}

		where $g\ge0$ is integrable and nonvanishing and doesn't depend on $L$. Furthermore, for each $\Lambda$, we can choose an $L(\Lambda)$ so that $\int_{\R}U(q_{\Lambda,L})=K$ and if $\Lambda\to 0$ then $L(\Lambda)\to\infty$. Thus, we have, for some $C>0$,
		\begin{align*}
				\frac{K}{T(q_{\Lambda, L(\Lambda)})} =\frac{{\displaystyle \int_{\ell}^\delta \bigg[(L-\xi)V(\frac{\Lambda\xi}{m(\xi)})k(\xi)  \bigg]d\xi} + C}{\frac{1}{2}\Lambda^2L(\Lambda)}.
		\end{align*}
		Using Taylor expansion on the left and right sides and letting $\Lambda\to0,$ we find that
		\begin{equation}\label{limit_T}
			K= \Nell V''(0) \lim_{\Lambda\to 0} T(q_{\Lambda,L(\Lambda)}) \ge \Nell V''(0)\limsup_{\eps\to 0}T_{K,\eps}.
		\end{equation}

		Combined with the inequality \eqref{ineq:lim}, we find that indeed
		\eqref{e:Tke-eq} holds.
\end{proof}

\begin{lemma}\label{reformulate2}
		The following are equivalent:
			\begin{enumerate}[i.]
				\item No minimizing sequence $\{ q^k\} \subseteq \cA_K$ vanishes, i.e., satisfies Lemma \ref{concentration-compactness}.ii.
				\item There exists $\eps(K)>0$ such that every minimizing sequence $\{q^k\}\subseteq \cA_K$ satisfies
				$$\liminf_{k\to\infty} \Vert [q^k]_{\ell}\Vert_\infty >\eps.$$
				\item $T$ satisfies the energy inequality
				\begin{equation}
						\Nell V''(0)\cdot T_K < K. \label{energy}\tag{E}
				\end{equation}
			\end{enumerate}
\end{lemma}
\begin{proof}
		$((ii)\Leftarrow (iii)).$ This is an easy consequence of Lemma \ref{ineq:liminf}.

		$((ii)\Rightarrow (iii)).$ Since we assume that $\Vert[q^k]_{\ell}\Vert_\infty >\eps$ for every
		minimizing sequence, necessarily from \eqref{e:Tke-eq},
		$$T_K < T_{K,\hat{\eps}} $$ for each $\hat{\eps}<\eps$. So, because $T_{K,\hat\eps}$
		increases as $\hat\eps$ decreases, we deduce
		$$T_K < \lim_{\hat{\eps}\to 0} T_{K,\hat{\eps}}= \frac{K}{\Nell V''(0)}.$$

		$((i)\implies (ii))$. Suppose we have that there exists a minimizing sequence $\{q^k\}$ with $\Vert[q^k]_{\ell}\Vert_\infty\to 0$.
		Then for all $y\in \R$ and $R>0$,

			\begin{align*}
				\int_{B_R(y)} U(q^k;\cdot)  &= \int_{B_R(y)} \int_{\ell}^\delta V\bigg(\frac{q^k(x+\xi)-q^k(x)}{m(\xi)} \bigg)k(\xi) \,d\xi \,dx\\
					&\le \int_{B_R(y)} C\int_{\ell}^\delta \bigg(\frac{q^k(x+\xi)-q^k(x)}{m(\xi)}\bigg)^2k(\delta) \, d\xi \, dx\\
					&\xrightarrow{k\to\infty} 0,
			\end{align*}

	which implies we have vanishing; i.e., (i) fails.

		$(ii)\implies (i)$.
		{\bf Step 1.} Let $q\in \cA_K$ such that $\Vert[q]_{\ell}\Vert_\infty>\eps$. We can then pick $y_0, \xi_0$ so that
		$$\bvert \frac{q(y_0+\xi_0)-q(y_0)}{m(\xi_0)}\bvert>\frac{\eps}{2}.$$

		Since $\xi_0\ge \ell$,
		$$|q(y_0+\xi_0)-q(y_0)| > \frac{m(\ell)\eps}{2}.$$
		By continuity, we can pick $\delta'$ so that $\forall \xi \in B_{\delta'}(\xi_0)$, $|q(y_0+\xi)-q(y_0)|>\frac{m(\ell)\eps}{4}.$ Without loss of generality, assume that $y+\xi_0+\delta' <\delta$ and $y_0+\xi_0-\delta' >\ell$. Thus,
		\begin{equation*}
				\begin{aligned}
						U(q;y_0)= & \int_{\ell}^\delta V(\frac{q(y_0+\xi)-q(y_0)}{m(\xi)}) k(\xi) \, d\xi \\
						\ge & \int_{y_0+\xi_0-\delta'}^{y_0+\xi_0+\delta'} V(\frac{m(\ell)\eps}{4m(\delta)})k(\ell) \, d\xi\\
						=& 2\delta'V(\frac{m(\ell)\eps}{4m(\delta)})k(\ell) >0.
				\end{aligned}
		\end{equation*}

		Thus, for each $\ell, \eps$, there exists a $C_{\ell,\eps}$ such that, for each $q$ such that $\Vert[q]_{\ell}\Vert_\infty>\eps$, there exists a $y_0$ such that
		\begin{equation}\label{est:Ubelow}
				U(q;y_0) \ge C_{\ell,\eps}.
		\end{equation}

		{\bf Step 2.} Let $\eps>0$, $C_{\ell,\eps}$ be as in step 1. Let $q$ be such that $\Vert[q]_{\ell}\Vert_\infty>\eps$. Thus, by step 1, there exists $y_0$ such that $U(q;y_0)>C_{\ell,\eps}$. By Lemma \ref{unif_cont} we then have, there exists an $R$ so that for $x\in B_R(y_0)$
				$$U(q;x) \ge \frac{C_{\ell,\eps}}{2}.$$
		Thus,
				$$\int_{B_R(y_0)} U(q;x) \ge RC_{\ell,\eps}.$$

		So, we have shown that for any $q$ such that $\Vert[q]_{\ell}\Vert_\infty>\eps$, there exist $R$ and $\eps_1$ dependent only on $\eps$ and $\ell$ so that for some $y_0$,
		$$\int_{B_R(y_0)}U(q;z)\, dz \ge \eps_1.$$
		This implies (i).
\end{proof}

\begin{proposition}\label{prop:energy-subadd}
	Let $U$ be as above and $K>0$ be fixed. Assume
	\begin{equation}
		T_K < T_\alpha + T_{K-\alpha}, \hspace{0.3in} \forall \alpha\in (0,K). \tag{S}
	\end{equation}
	and
	\begin{equation}
		\Nell V''(0)\cdot T_K < K. \tag{E}
	\end{equation}
	Then there exists a minimizer of $T$ on $\cA_K$.
\end{proposition}
\begin{proof}
		Let $\{q^k\}\subseteq \cA_K$ be a minimizing sequence of $T$. By lemmas \ref{reformulate1} and \ref{reformulate2}, we have that, up to a subsequence, $\{q^k\}$ satisfies the compactness case of the concentration-compactness lemma. By replacing $q^k$ by $q^k(y_k+\cdot) -q^k(y_k)$, we can assume that $q^k$'s have centers at $0$ and $q^k(0)=0$. Since $\{q^k\}$ is bounded in $H$, there exists a subsequence that weakly converges to $q\in H$. By weak lower semicontinuity of norm, $$T(q)\le \liminf T(q^k)= \inf_{\cA_K}T.$$
		To see that $\Eell(q)=K$, we note that $W^{1,2}(B_R(0))$ is compactly embedded in $L^\infty(B_R(0))$. Therefore, $q^k\to q$ in $L^\infty(B_{R}(0))$.  This implies $U(q^k;\cdot)\to U(q;\cdot)$ in $L^\infty(B_{R-\delta}(0))$. Therefore,
		$$\int_{B_R(0)}U(q^k;\cdot) \to \int_{B_R(0)} U(q;\cdot)$$ for all $R>0$.

		On the other hand, by Lemma \ref{concentration-compactness} case i, for every $\eps'>0$ there exists an $R$ such that for every $k\in \N$
		$$\int_{\R\backslash B_{R}} U(q^k;\cdot) \le \eps'.$$
		Using a standard $\eps/2$ argument, we see that
		$$K=\lim \int_\R U(q^k;\cdot) = \int_\R U(q;\cdot).$$
		So, $q$ is a minimizer of $T$ on $\cA_K$.
\end{proof}

\begin{lemma}\label{lem:energy}
		There exists $K_0>0$ such that for all $K>K_0$, (\ref{energy}) holds.
\end{lemma}
\begin{proof}
		 Fix $\Lambda>0$. Let $q_{\Lambda, L}$ be defined as above. Then, from (\ref{eq:qlambdal}),
		\begin{align*}
			\Eell(q_{\Lambda,L}) &\ge \int_{\ell}^\delta \bigg[(L-\xi)V(\frac{\Lambda\xi}{m(\xi)})k(\xi)   \bigg]d\xi\\
				&> \int_\ell^\delta \bigg[(L-\xi)\frac{1}{2}V''(0)\frac{\Lambda^2\xi^2}{m(\xi)^2}   k(\xi)  \bigg]d\xi   \\
				&= \frac{1}{2}L\Lambda^2 V''(0)\int_\ell^\delta \bigg[ \frac{\xi^2 k(\xi)}{m(\xi)^2} - \frac{\xi^3 k(\xi)}{m(\xi)^2L}\bigg] d\xi\\
					&= T(q_{\Lambda,L})V''(0)\int_\ell^\delta \bigg[ \frac{\xi^2 k(\xi)}{m(\xi)^2} - \frac{\xi^3 k(\xi)}{m(\xi)^2L}\bigg] d\xi
		\end{align*}
		by superquadratic property of $V$.

		So, there exists $L_0$ such that $\forall L>L_0$,
		$$\Eell(q_{\Lambda, L})> \Nell V''(0)T(q_{\Lambda,L}). $$
		Let $K_0 := \Eell(q_{\Lambda,L_0})$. Since
		$L\mapsto\Eell(q_{\Lambda,L})$  is increasing,
		for each $K>K_0$, there is an $L>L_0$ so that $K=\Eell(q_{\Lambda,L})$, hence
		$$K>\Nell V''(0)T_K.$$
\end{proof}

\begin{lemma}\label{lem:subadditivity}
		Let $K_0$ be as in Lemma \ref{lem:energy}, then for all $K>2K_0$, (\ref{subadditivity}) holds.
\end{lemma}
\begin{proof}
		As Lions pointed out in \cite{L}, to prove that a function $h:[0,K]\to\R$ satisfies $h(K)<h(\alpha)+ h(K-\alpha)$ for all $\alpha\in (0,K)$, it suffices to show the following:
		\begin{equation}\tag{$ \tilde{S}$}
				\begin{cases}
						h(\theta \alpha)\le \theta h(\alpha), & \forall \alpha\in (0,\frac{K}{2}), \forall \theta\in(1,\frac{K}{\alpha}],\\
						h(\theta\alpha) < \theta h(\alpha), &\forall \alpha\in[\frac{K}{2},K), \forall \theta \in (1,\frac{K}{\alpha}].
				\end{cases}
		\end{equation}

		We want to check that $h(\alpha)=T_\alpha$ satisfies the above properties.

		Let $\alpha\in (0,K)$ and $\theta\in (1,K/\alpha]$. We first consider the case $\alpha\ge K/2$. By the concentration-compactness principle and Lemma \ref{lem:energy}, there exists $\eps$ and $C$ such that
		$$T_{\alpha} = \inf\{T(q):q\in \cA_{\alpha,\eps,C}\} $$
		where $$\cA_{\alpha,\eps,C}:= \{q\in H: \Eell(q)=\alpha,\ \Vert [q]_{\ell}\Vert_{\infty}\ge \eps,\ \Vert q\Vert \le C \}. $$
		By (\ref{est:Ubelow}), there exists $\alpha_0>0$ such that, for $A:={\{z:[q(z)]_{\ell}\ge\eps/2\}}$,
		$$\int_A U(q;z) \ge \alpha_0 $$ for all $u\in \cA_{\alpha,\eps,C}$.

		 Let $q\in \cA_{\alpha,\eps,C}$. Since $\Eell(q)=\alpha$ and for $\lambda=\sqrt{\theta}$,
		$$\Eell(\lambda q) = \Eell(\sqrt{\theta} q) \ge \theta\Eell(q)=\theta \alpha.$$ By the intermediate value theorem, there exists $\lambda=\lambda(\theta,q)\in [1,\sqrt{\theta}]$ such that
		$$\Eell(\lambda q)=\theta \alpha.$$
		We claim that $\lambda<\sqrt{\theta}$. To see this, suppose by contradiction, $\lambda=\sqrt{\theta}$. Let $$\theta_0:= \min\bigg\{\frac{V(\lambda r)}{\lambda^2 V(r)}: \vert r\vert \in [\frac{\eps}{2},C], \lambda\in [\frac{1+\sqrt{\theta}}{2},\sqrt{\theta}] \bigg\}.$$

		Note that $\theta_0>1$.
		We then have
		\begin{equation*}
				\begin{aligned}
						\theta\alpha &= \Eell(\sqrt{\theta} q) = \int_{A^C} U(\sqrt{\theta} q;\cdot) +\int_{A} U(\sqrt{\theta} q;\cdot)\\
								&\ge \theta\int_{\R\backslash A} U(q;\cdot) +\theta_0\theta\int_{A} U(q;\cdot)\\
								&\ge \theta[\alpha + (\theta_0-1)\alpha_0],
				\end{aligned}
		\end{equation*}
		which is a contradiction.

		Define
		$$\lambda_0^2:= \frac{\theta \alpha}{\alpha + (\theta_0-1)\alpha_0}<\theta.$$
		The same calculation above shows that
		$$\lambda(\theta,q) \le \lambda_0. $$
		Thus,
		\begin{equation*}
				\begin{aligned}
						T_{\theta\alpha}    &\le \inf\bigg\{T(\lambda(\theta,q)q):q\in \cA_{\alpha,\eps,C} \bigg\}\le \lambda_0^2\inf\bigg\{T(q): q\in \cA_{\alpha,\eps,C} \bigg\}\\
						& =\lambda_0^2T_\alpha <\theta T_\alpha.
				\end{aligned}
		\end{equation*}
\end{proof}
Combining Proposition \ref{prop:energy-subadd}, Lemma \ref{lem:energy} and Lemma \ref{lem:subadditivity}, we have shown the following:
\begin{proposition}
		There exists a $K_0$ such that for all $K> K_0$, there exists $q\in \cA_K$ that minimizes the problem
		$$\min \bigg\{ T(q): q\in \cA_K \bigg\}.$$
\end{proposition}

\section{Properties of minimizers} \label{properties}
We now study the properties of minimizers $q=q^\ell$ for the truncated problem (Problem \ref{approximate}) where $\ell \in (0,\delta)$ is fixed. To make the notations less cluttered, we continue to suppress explicit dependence on $\ell$ throughout this section.

\subsection{Euler-Lagrange equation\label{E-L}}
This subsection is devoted to showing that the Euler-Lagrange equation of the truncated minimization problem (Problem \ref{approximate})  is equation (\ref{eq:travelling-wave-cutoff}).

\begin{proposition}
	Let $q\in \cA_K$ be such that
	$$T(q)=\inf_{\cA_K} T.$$ Then $q$ satisfies (\ref{eq:travelling-wave-cutoff}) and $q\in C^2$.
\end{proposition}

\begin{proof}

	Let $\zeta\in H$ be a non-zero function. Define $\Psi, \Phi:\R^2\to \R_+$ by
	$$ \Psi(t,\eps):= \Eell(q+tq +\eps \zeta),\qquad
	\Phi(t,\eps) : = T(q+tq+\eps\zeta) .$$

First, we claim that
$$\frac{\partial \Psi}{\partial \eps}(0,0) = \int_{\R}\int_{\ell}^\delta V'\bigg(\frac{q(z+\xi)-q(z)}{m(\xi)}\bigg)(\zeta(z+\xi)-\zeta(z))\frac{k(\xi)}{m(\xi)}\, d\xi \, dz. $$

To see this,
let $\eta_1(z,\xi) := (q(z+\xi)-q(z))/m(\xi)$
and $\eta_2(z,\xi):= (\zeta(z+\xi) -\zeta(z))/m(\xi)$.
We only need to find a function in $L^1(\R)$ that dominates the quantity
\begin{equation}\label{eq:dominated}
\frac1{\eps} \left( V(\eta_1+\eps\eta_2) - V(\eta_1)\right)
\end{equation}
to be able to pass the derivative inside.
Note that $\eta_1(\cdot,\xi)$ and $\eta_2(\cdot,\xi)$ are uniformly bounded
in $L^2(\R)$, because, for example, by Cauchy-Schwarz,
\[
\int_\R |\zeta(z+\xi)-\zeta(z)|^2 \,dz \le  \int_\R \xi\int_0^\xi \zeta'(z+s)^2\,ds\,dz
= 2\xi^2 T(\zeta).
\]
Because $|V'(t)|\le C_1|t|$ whenever $|t|\le \|\eta_1\|_\infty+\eps \|\eta_2\|_\infty$,
where $C_1$ is a local bound for $V''$, we then find that whenever $|\eps|\le1$,
\begin{align*}
|V(\eta_1+\eps\eta_2)-V(\eta_1)| =
|V(|\eta_1+\eps\eta_2|)-V(|\eta_1|)|
\le C_1(|\eta_1|+|\eta_2|) \eps|\eta_2| \,.
\end{align*}
%
Thus, (\ref{eq:dominated}) is dominated by
$C(|\eta_1|+|\eta_2|)|\eta_2| \in L^1(\R)$.

Similarly, \begin{equation}
		\frac{\partial\Psi}{\partial t}(0,0) = \int_{\R}\int_{\ell}^{\delta} V'(\frac{q(z+\xi)-q(z)}{m(\xi)})(q(z+\xi)-q(z))\frac{k(\xi)}{m(\xi)} \, d\xi \, dz.
\end{equation}
We also have that, since $V(x)$ is strictly increasing for $x>0$ and strictly decreasing for $x<0$, $V'(x)x>0$ for $x\not=0$. Therefore,  $\frac{\partial \Phi}{\partial t}(0,0) >0$ and so
\begin{equation}
		\nabla \Psi(0,0)\not=0.
\end{equation}

On the other hand, it is easy to see that
$$\frac{\partial \Phi}{\partial \eps}(0,0) = \int_{\R} q'\zeta' \,,\qquad
\frac{\partial \Phi}{\partial t}(0,0) = \int_{\R} (q')^2.$$

We rewrite the problem in terms of minimizing $\Phi$ under the constraint $\Psi=K$. By the hypothesis, we know that $\Phi(0,0)$ is a minimizer of $\Phi$ under the constraint. Because of $\nabla\Psi(0,0)\not=0$, we can apply the Lagrange multiplier rule to deduce that there exists a $\lambda$ such that $(0,0)$ is a critical point of
$\Phi-\lambda\Psi$.
That means
$$0=\nabla \Phi(0,0) -\lambda\nabla\Psi(0,0)$$
In particular, we have
\begin{gather}
	\int_\R \bigg[ q'\zeta' - \lambda\int_{\ell}^{\delta} V'(\frac{q(z+\xi)-q(z)}{m(\xi)})(\zeta(z+\xi)-\zeta(z))\frac{k(\xi)}{m(\xi)} \, d\xi \bigg]dz=0 \label{distribution_der}\\
	\int_\R\bigg[ (q')^2 - \lambda \int_{\ell}^{\delta} V'(\frac{q(z+\xi)-q(z)}{m(\xi)})(q(z+\xi)-q(z))\frac{k(\xi)}{m(\xi)} \, d\xi \bigg]dz=0. \label{distribution_second}
\end{gather}

Equation (\ref{distribution_second}) tells us that $\lambda$ is greater than 0 and independent of $\zeta$.
Rewriting (\ref{distribution_der}), for any $\zeta \in C_c^\infty(\R)$ we have
\begin{align*}
		\int_{\R} q'\zeta' dz 
		= \lambda \int_\R \int_{\ell}^\delta \bigg[f(q(z) - q(z-\xi),\xi) - f(q(z+\xi)-q(z),\xi) \bigg] d\xi \, \zeta(z) \, dz.
\end{align*}

 So, $-\lambda \int_{\ell}^\delta \bigg[f(q(z) - q(z-\xi),\xi) - f(q(z+\xi)-q(z),\xi) \bigg] d\xi $ is the distributional second derivative of $q$. But since $\lambda U'(q(z))$ is continuous, $q''$ exists in classical sense and is continuous. So, $q\in C^2(\R)$. This means that $q$ solves (\ref{eq:travelling-wave-cutoff}) where $c=\lambda^{-1/2}$.
\end{proof}

\subsection{Monotone travelling wave in the approximate problem \label{mon-existence}}
\begin{lemma}\label{lem:monotone}
		Let $V$ be as above. Then, any minimizer $q$ of $T$ over $\cA_K$ is monotone.
\end{lemma}
\begin{proof}
		This proof is done by applying a scaling argument repeatedly.
		Let $q\in \cA_K$ be a minimizer of $T$. Consider the function
		$$\tilde{q}(z):= \int_0^z \vert q'(s)\vert\,ds.$$
		We claim that for all $z\in \R$ and $\xi\in [\ell,\delta]$,
\begin{equation}\label{e:tildeqdiff}
\tilde{q}(z+\xi)-\tilde{q}(z) = \vert q(z+\xi)-q(z)\vert.
\end{equation}
		We only need to show $\tilde{q}(z+\xi)-\tilde{q}(z) \le \vert q(z+\xi)-q(z)\vert$ since the other inequality is obvious by definition. Suppose there exists a $z_0$ and $\xi_0$ so that $$\tilde{\eta}(z_0,\xi_0):=\tilde{q}(z_0+\xi_0)-\tilde{q}(z_0) > \vert q(z_0+\xi_0)-q(z_0)\vert=:\eta(z_0,\xi_0).$$
		Since $\tilde{\eta}$ and $\eta$ are continuous, there exists a $\delta'>0$ such that on $A:=(z_0-\delta',z_0+\delta')\times(\xi_0-\delta',\xi_0+\delta')$,
		$$\tilde{\eta}(z,\xi) >\eta(z,\xi).$$
		Without loss of generality, assume that $\xi_0-\delta'>\ell$. Thus, on $A$, we have that
		$$V(\frac{\tilde{q}(z+\xi)-\tilde{q}(z)}{m(\xi)}) > V(\frac{q(z+\xi)-q(z)}{m(\xi)}). $$

		and, therefore,
		$$\Eell(\tilde{q})>\Eell(q).$$

		So, there exists $\lambda\in (0,1)$ such that
		$$\Eell(\lambda \tilde{q}) = \Eell(q)=K.$$
		But, we have
		$$T(\lambda \tilde{q}) = \lambda^2 T(\tilde{q})< T(\tilde{q}) =T(q),$$
		which contradicts the optimality of $q$.
	 This proves the claim \eqref{e:tildeqdiff}.

		Now, suppose $q$ is not monotone. Since $q\in C^2$, there exists $a,b\in \R$ such that $q'(a)<0$ and $q'(b)>0$.
Without loss we assume $a<b$. Since, for $z\in \R$ and $\xi\in[\ell,\delta]$,
		$$\int_{z}^{z+\xi} \vert q'\vert =\bigg\vert\int_{z}^{z+\xi}q' \bigg\vert, $$
on each interval $[z,z+\delta]$, either $q'\ge 0$ or $q'\le 0$. We infer that $b-a\ge \delta$.
Let
\[
z_1= \sup\{z\in(a,b):q'(z)<0\},\qquad
z_2 = \inf\{z\in(z_1,b):q'(z)>0\}.
\]
Then $z_2-z_1\ge\delta$ and $q'(z)=0$ for all $z\in(z_1,z_2)$.
		Define
		$$\tilde{\tilde{q}}(z):=\begin{cases}
			\tilde{q}(z), & z\le z_1 \\
			\tilde{q}(z+(z_2-z_1)), & z>z_1
		\end{cases}. $$
		(The basic idea is we want to get rid of this constant part of $\tilde{q}$.)
		We have that

		\begin{align*}
 &     \Eell(\tilde{\tilde{q}}) - \Eell(\tilde{q}) = \\
				&=\int_{\ell}^{\delta} \bigg[\int_\R V\bigg(\frac{\tilde{\tilde{q}}(z+\xi)-\tilde{\tilde{q}}(z)}{m(\xi)} \bigg) -V\bigg(\frac{\tilde{q}(z+\xi)-\tilde{q}(z)}{m(\xi)}\bigg) dz\bigg]k(\xi) \, d\xi \\
				&= \int_{\ell}^\delta \bigg[\int_{z_1-\xi}^{z_1} V\bigg(\frac{\tilde{q}(z+\xi+(z_2-z_1)) -\tilde{q}(z)}{m(\xi)} \bigg)dz \\
					&\hspace{1cm} -\int_{z_1-\xi}^{z_1} V\bigg(\frac{\tilde{q}(z_1)-\tilde{q}(z)}{m(\xi)} \bigg)dz -\int_{z_2-\xi}^{z_2} V\bigg(\frac{\tilde{q}(z+\xi)-\tilde{q}(z_1)}{m(\xi)}\bigg)dz\bigg]k(\xi) \, d\xi \\
				&= \int_{\ell}^\delta \int_{z_1-\xi}^{z_1}\bigg[
 V\bigg(\frac{\tilde{q}(z+\xi+(z_2-z_1)) -\tilde{q}(z)}{m(\xi)} \bigg)
\\ &\hspace{2cm}
- V\bigg(\frac{\tilde{q}(z_1)-\tilde{q}(z)}{m(\xi)} \bigg)
-V\bigg(\frac{\tilde{q}(z+\xi +(z_2-z_1))-\tilde{q}(z_1)}{m(\xi)} \bigg)\bigg]dz\,k(\xi)\, d\xi \\
					&\ge 0.
		\end{align*}
		Here, due to the strict convexity of $V$, the equality occurs if and only if for all $(z,\xi)\in [z_1-\xi,z_1]\times[\ell,\delta]$ either $$r_1(z,\xi):=\tilde{q}(z_1)-\tilde{q}(z)=0$$ or $$r_2(z,\xi):=\tilde{q}(z+\xi+(z_2-z_1))-\tilde{q}(z_1)=0.$$

		By the way 
we define $z_1, z_2$, neither of these conditions can hold. Thus,
		$$\Eell(\tilde{\tilde{q}})> \Eell(\tilde{q}). $$
		Again, let $\lambda\in(0,1)$ so that $\Eell(\lambda\tilde{\tilde{q}})=\Eell(\tilde{q})=K.$ But then,
		$$T(\lambda \tilde{\tilde{q}})=\lambda^2 T(\tilde{\tilde{q}}) < T(\tilde{\tilde{q}})=T(\tilde{q})=T(q),$$
		which is a contradiction.
\end{proof}

\begin{proof}[Proof of Theorem~\ref{thm1}]
		Combine Lemma \ref{lem:monotone}, the Euler-Lagrange equation, and the argument in subsection \ref{sub:existence}, we achieve the desired result.
\end{proof}

\begin{remark}\label{sign} {\it If $q$ is a solution to the approximating symmetrized problem, then $-q$ is also a solution. We will assume $q$ to be monotone increasing to make the subsequence analysis notationally easier.}
\end{remark}

\section{Passing to the limit: proof of Theorem~\ref{thm2}} \label{limiting}

Now we turn our attention to step 3 in the strategy outlined in Subsection \ref{strategy}. Recall $\cE^{\ell}$ in \eqref{potential_energy} is the potential energy associated with
cutoff parameter $\ell$. Our goal is to extract a subsequential limit from the minimizers $q^\ell$ of Problem \ref{approximate}.

We need an improved version of Lemma \ref{lem:energy}. Recall $T^\ell_K$ is defined in \eqref{kinetic_energy_K}.
		\begin{lemma}\label{lem:energy_revisit}
				For some $\ell_0, K_0>0$, there exists a $C>0$ independent of $\ell$ (and $K_0$) so that for all $K>K_0$
and all $\ell\in(0,\ell_0)$,
				$$K - C> \Nell V''(0)T_K^{\ell}\,,
\quad\mbox{where}\quad \Nell = \int_\ell^\delta \frac{\xi^2 k(\xi)}{m(\xi)^2}\, d\xi\,. $$
		\end{lemma}

		\begin{proof}
				Fix $\lambda$. Let $L^{\ell} > 3\delta$. Recall from the argument following equation (\ref{eq:qlambdal}) that
						\begin{align*}
								\int_{\R} U(q_{\Lambda,L^{\ell}};\cdot)
&=\int_{\ell}^\delta \bigg[(L^{\ell}-\xi)V(\frac{\Lambda\xi}{m(\xi)})k(\xi) + g(\xi)  \bigg]d\xi\\
										&= \int_{\ell}^\delta \bigg[(L^{\ell}-\xi)V(\frac{\Lambda\xi}{m(\xi)})k(\xi)   \bigg]d\xi + C_\ell\,,
						\end{align*}
where $C_\ell>0$.
						By the way we choose $L^{\ell}$, $g(\xi)$ is independent of $\ell$ and therefore $C_\ell>C_{\ell_0}>0$ for $0<\ell<\ell_0$. Following the same analysis as in Lemma \ref{lem:energy}, we have that
						\begin{equation}
								\cE(q_{\Lambda, L^{\ell}}) - {C_{\ell_0}}> T(q_{\Lambda, L^\ell}) V''(0)\int_\ell^\delta \bigg[ \frac{\xi^2 k(\xi)}{m(\xi)^2} - \frac{\xi^3 k(\xi)}{m(\xi)^2L^{\ell}}\bigg] d\xi.
						\end{equation}
						Again, following the same analysis as in Lemma \ref{lem:energy}, our conclusion holds.
		\end{proof}

		 Pick $K_0$ as in Lemma \ref{lem:energy_revisit}. For small enough $\ell_0$, we have that for $K>2K_0$ and for all $\ell\in(0, \ell_0]$, there exists a minimizer $q^{\ell}$ of $T$ over the constraint $\cE^{\ell}(q)=K$.

\begin{proposition}
		For $0<\ell_1<\ell_2<\ell_0$, $T(q^{\ell_1})\le T(q^{\ell_2})$.
\end{proposition}

\begin{proof}
For each $\ell$, let $q^{\ell}$ be a minimizer of $T^{\ell}$ under the constraint $\cE^{\ell}(q) = K$.
Then
\begin{align*}
	\cE^{\ell_2}(q^{\ell_2}) &= \int_\R \int_{\ell_2}^\delta V(\frac{q^{\ell_2}(z+\xi) - q^{\ell_2}(z)}{m(\xi)}) k(\xi) \, d\xi\, dz \\
		&\le \int_\R \int_{\ell_1}^\delta V(\frac{q^{\ell_2}(z+\xi) - q^{\ell_2}(z)}{m(\xi)}) k(\xi) \, d\xi\, dz \\
        &= \cE^{\ell_1}(q^{\ell_2})=: K^*\ge K.
\end{align*}
Note that from Lemma \ref{mon}, we have that for each fixed $\ell$, $K\mapsto T_K^{\ell}$ is monotone increasing.
Therefore, $$T(q^{\ell_2}) \ge T_{K^*}^{\ell_1} \ge T_K^{\ell_1} = T(q^{\ell_1}).$$
\end{proof}
We infer that $\{q^{\ell}\}$ is a bounded sequence in $H$ and so there exists a subsequence $\{q^{\ell_i}\}$ that weakly converges to some $q\in H$. In particular, for all $\zeta\in H$,
$$\int_\R (q^{\ell_i})'\zeta' \xrightarrow{i\to\infty}\int_\R q' \zeta'.$$
Because $q^\ell(0)=0$ it is straightforward to infer that $\{q^{\ell_i}\}$ also converges
weakly  to $q$ in $W^{1,2}([a,b])$ for every $a\le b\in \R$.

We want next to show $q$ is a non-trivial minimizer of $T$ subject to $\cE^0(q)=K$.

We introduce some notation to make the following analysis a little more tractable. Write
$$g(x):= c_1\vert x\vert^{\gamma_1} + c_2\vert x\vert^{\gamma_2}$$ where $c_i, \gamma_i$ ($i\in \{1,2\}$) are from Hypothesis (H\ref{A1}). We note that, by Hypothesis (H\ref{A2}),
$$\int_0^\delta g\bigg(\frac{1}{m(\xi)}\bigg)\frac{\xi^2 k(\xi)}{m(\xi)^{2}} \, d\xi <\infty.$$

	\begin{proposition}\label{p:recenter}
		Up to re-centering and taking the limit of $\{q^{\ell_i}\}$ again, $q$ is non-trivial.
	\end{proposition}

	\begin{proof}


			Combine Hypothesis (H\ref{A1}) and Lemma \ref{lem:energy_revisit}, we have that there exists $C>0$ so that, for every $\ell_i$,
			\begin{align*}
				K = & \int_\R \int_{\ell_i}^\delta V\bigg(\frac{q^{\ell_i}(z+\xi) - q^{\ell_i}(z)}{m(\xi)}\bigg) k(\xi) \, d\xi \, dz\\
						\le& \int_\R \int_{\ell_i}^\delta \frac{1}{2} \bigg(V''(0)+g\bigg(\frac{q^{\ell_i}(z+\xi) - q^{\ell_i}(z)}{m(\xi)}\bigg) \bigg)\bigg(\int_0^1 (q^{\ell_i})'(z+\xi s)^2 ds \bigg) \frac{\xi^2 k(\xi)}{m(\xi)^{2}} \, d\xi\, dz \\
						\le&\ V''(0)T(q^{\ell_i})\int_0^\delta \frac{\xi^2 k(\xi)}{m(\xi)^{2}} \, d\xi
						+\sup_{z\in \R} g(q^{\ell_i}(z+\delta) - q^{\ell_i}(z))
						T(q^{\ell_i})\int_0^\delta g\bigg(\frac{1}{m(\xi)}\bigg)\frac{\xi^2 k(\xi)}{m(\xi)^{2}} \, d\xi
\\
						\le&\  K-C+ \sup_{z\in \R} g(q^{\ell_i}(z+\delta) - q^{\ell_i}(z))  T(q^{\ell_0})\int_0^\delta g\bigg(\frac{1}{m(\xi)}\bigg)\frac{\xi^2 k(\xi)}{m(\xi)^{2}} d\xi
			\end{align*}
			Thus,
			\begin{equation}
				\sup_{z\in \R} (q^{\ell_i}(z+\delta)-q^{\ell_i}(z)) \ge  C_* >0
			\end{equation}
			where
\[
g(C_*) T(q^{\ell_0})\int_0^\delta g\bigg(\frac{1}{m(\xi)}\bigg)\frac{\xi^2 k(\xi)}{m(\xi)^{2}} d\xi \, > \frac{C}{2} >0.
\]

Now we can re-center $q^{\ell_i}$ to get
			$q^{\ell_i}(\delta)-q^{\ell_i}(0) > \frac12{C_*} >0$
			for all $\ell_i$. Take the weak limit again to arrive at a non-trivial limit.
	\end{proof}

We now need to show that $q$ satisfies (\ref{eq:limiting}) for some $c\not=0$.

		\begin{lemma}
				 For some $C>0$, $C\le\lambda_{\ell_i}\le T(q^{\ell_0})/K$ for all $i$.
		\end{lemma}
		\begin{proof}
		To see the lower bound on $\lambda_i$, we note that $\int ({q^{\ell_i}}')^2$ is decreasing as $\ell_i\to 0$ and that $q$ is non-trivial so that
		$$\liminf_{i\to\infty} \int({q^{\ell_i}}')^2 \ge \int (q')^2 >0.$$

By similar estimates as in the proof of Proposition~\ref{p:recenter}, we find
\begin{align*}
				 \int_\R \int_{\ell_i}^\delta V'
\bigg(\frac{q^{\ell_i}(z+\xi) - q^{\ell_i}(z)}{m(\xi)}\bigg)
\bigg(\frac{q^{\ell_i}(z+\xi) - q^{\ell_i}(z)}{m(\xi)}\bigg)
k(\xi) d\xi \, dz\le C,
\end{align*}
where $C<\infty$ is independent of $i$.
The lower bound on $\lambda_{\ell_i}$ follows by the last two estimates
and the identity (\ref{distribution_second})
with $q$, $\lambda$ replaced by $q^{\ell_i}$, $\lambda_{\ell_i}$.

	To prove the upper bound on $\lambda_{\ell_i}$, we note that, by Taylor theorem and monotonicity of $V$, we have
	$$V'(\frac{q(z+\xi)-q(z)}{m(\xi)})(q(z+\xi)-q(z)) \ge V(\frac{q(z+\xi)-q(z)}{m(\xi)}).$$
The integral over $\xi$ and $z$ gives $K$ on the right hand side, so
	combining this with \eqref{distribution_second}, we get the desired bound.
		\end{proof}
		Due to the bounds in the last lemma, there exists a $\lambda_0>0$ so that, up to a subsequence, $\lim \lambda_i = \lambda_0$.

		We now want to show that the limiting function $q$ satisfies equation~\eqref{eq:limiting}. We need to look at (\ref{distribution_der}) for the proof of this one.

		\begin{proposition}
The limit $q$ is a solution to equation (\ref{eq:limiting})
with $c^2=1/\lambda_0$. Furthermore, $c^2 > c_0^2 = V''(0) \int_0^\delta \frac{k(\xi)\xi^2}{m(\xi)^2} \, d\xi$.
		\end{proposition}
		\begin{proof}
	 Fix $\zeta\in C_c^\infty(\R)$ and let $M>0$ be such that $\supp (\zeta(z+\delta) -\zeta(z)) \subseteq (-M,M)$. Since $\zeta$ is smooth and compactly supported, $|\zeta(z+\xi) -\zeta(z)| \le C\xi$ for some $C>0$. By the fact that $q^{\ell_i}$ weakly converges to $q$ in $H$, we have
		\begin{align*}
			\int_{\R} q'(z)\zeta'(z) dz =& \lim_{i\to\infty} \int_\R  {q^{\ell_i}}'(z)\zeta'(z)\, dz\\
				=& \lim_{i\to\infty}\lambda_{\ell_i} \int_{\ell_i}^{\delta} \bigg[\int_{-M}^M V'(\frac{q^{\ell_i}(z+\xi)-q^{\ell_i}(z)}{m(\xi)})(\zeta(z+\xi)-\zeta(z)) \frac{k(\xi)}{m(\xi)}\, dz \bigg]d\xi \\
						=&\ \lambda_0\int_{\R}\bigg[\int_0^\delta V'(\frac{q(z+\xi)-q(z)}{m(\xi)})(\zeta(z+\xi)-\zeta(z)) \frac{k(\xi)}{m(\xi)} \, d\xi \bigg]dz.
		\end{align*}
		The last equality holds by the dominated convergence theorem, since for each $\xi$, we have
		\begin{align*}
			&\int_{-M}^M V'(\frac{q^{\ell_i}(z+\xi) - q^{\ell_i}(z)}{m(\xi)})|\zeta(z+\xi)-\zeta(z)| \frac{k(\xi)}{m(\xi)} \, dz \\
			 \le &\ C\int_{-M}^M \bigg(V''(0) + g\bigg( \frac{q^{\ell_i}(z+\xi) - q^{\ell_i}(z)}{m(\xi)}\bigg)\bigg) \bigg(\frac{\xi \cdot \int_0^1 (q^{\ell_i})'(z+s\xi) ds}{m(\xi)}\bigg)\frac{\xi k(\xi)}{m(\xi)^2} \, dz\\
			 \le &\ C\bigg(V''(0) + g\bigg( \frac{(\delta T(q^{\ell_0}))^{1/2}}{m(\xi)}\bigg)\bigg) \frac{\xi^2 \, k(\xi)}{m(\xi)^2} \int_0^1 \int_{-M-\delta}^{M+\delta} (q^{\ell_i})'(\hat{z}) \, d\hat{z}\,ds \\
			 \le &\ C\bigg(V''(0) + g\bigg( \frac{(\delta T(q^{\ell_0}))^{1/2}}{m(\xi)}\bigg)\bigg) \frac{\xi^2\, k(\xi)}{m(\xi)^2} (2(M+\delta))^{1/2} (T(q^{\ell_0}))^{1/2} \\
			 = &\ C_1 \frac{\xi^2 \, k(\xi)}{m(\xi)^2} + C_1 g\bigg( \frac{C_2}{m(\xi)}\bigg)\frac{\xi^2 \, k(\xi)}{m(\xi)^{2}},
		\end{align*}
and this is integrable over $\xi\in(0,\delta)$ by Hypothesis~(H\ref{A2}).

To see that $c^2 > c_0^2$, we notice that, by the previous lemma, $1/c^2=\lambda_0 \le T(q^{\ell_0})/K$. Therefore, by Lemma \ref{lem:energy_revisit}, 
\begin{equation*}
    c^2 \ge \frac{K}{T(q^{\ell_0})} > N_\ell V''(0) = V''(0)\int_0^\delta \frac{\xi^2 k(\xi)}{m(\xi)^2}\, d\xi = c_0^2.
\end{equation*}
	\end{proof}  
    
		\begin{proposition}
			$q$ is a minimizer of problem \ref{generalization} with the symmetrized potential $\overline{W}$
		\end{proposition}
		\begin{proof}
			We proceed by contradiction. Suppose that $q$ is not a minimizer of problem \ref{generalization} with the symmetrized potential $\overline{W}$.
			This means, there exists a function $p\in H$ so that $\cE(p)= K$ and
				$$T(p) < T(q)\le T(q^\ell)$$ for all $\ell <\ell_0$.


Let $\eps_2=T(q)-T(p)>0$. Pick $\ell>0$ small enough,
			$$\frac{K-\cE^\ell(p)}{N_\ell V''(0)} < \eps_2.$$
			Let $\eps_1:=K-\cE^\ell(p) >0$.
			By the same argument as in remark \ref{technical}, we can assume that $p'$ has compact support. Then, let $x_0=2\delta+\sup (\supp p')$ and define a function $\overline{p}$ by
			\begin{equation*}
				\overline{p}(x):= \begin{cases}
						p(x), \hspace{3cm} x\le x_0, \\
						q_{\Lambda,L}(x-x_0) + p(x_0), \hspace{0.5cm} x> x_0,
				\end{cases}
			\end{equation*}
			where we pick $L=L(\Lambda)$ for small $\Lambda$ so that
$\cE^{\ell}(q_{\Lambda, L(\Lambda)}) =\eps_1$.
. Now, we compute
			\begin{align*}
					\cE^\ell(\overline{p}) &= \int_{\R} \int_\ell^\delta V\bigg(\frac{\overline{p}(z+\xi) - \overline{p}(z)}{m(\xi)} \bigg) k(\xi) \, d\xi\, dz = \cE^\ell(p) + \cE^\ell(q_{\Lambda,L(\Lambda)}) = K. \\
			\end{align*}
			Now, by the first equality in \eqref{limit_T} (with $K$ replaced by $\eps_1$), for sufficiently small $\Lambda$ we have 
			$$ \frac{\eps_1}{N_\ell V''(0)}\approx T(q_{\Lambda, L(\Lambda)}) < \eps_2. $$
			But then,
			$$T(\overline{p}) = T(p)+ T(q_{\Lambda, L(\Lambda)}) < T(q^\ell),$$
			which is a contradiction.
		\end{proof}

		Now, by remark \ref{sign} both $q$ and $-q$ are solution to \eqref{eq:limiting}. So, if the original potential (which has not been symmetrized) is superquadratic on $[0,\infty)$, then $q$ would be a function that satisfies the corresponding (original) travelling wave equation since $q(z+\xi) - q(z)\ge 0$ (which belongs to the domain for the superquadratic portion of the potential).

		Similarly, if the original potential is superquadratic on $(-\infty, 0]$, $\tilde{q}=-q$ would be the function that satisfies the corresponding travelling wave equation since $\tilde{q}(z+\xi) - \tilde{q}(z) \le 0$.

		Combining everything in this subsection, we have proven Theorem~\ref{thm2}.

\section{Low energy \label{low-energy}}

In this section, we prove Theorem~\ref{thm:low-energy}, showing existence of travelling waves for Silling's model in \cite{S} with low potential energy.
(Recall from remark \ref{r:Silling}.
this corresponds to the choice $m(\xi)=\xi$, $k(\xi)=|\xi|$ with symmetrized potential
$V(s) = \frac12 s^2(1+\frac13|s|)$.)
This is done by revisiting the previous calculations with a different approximating function. Namely, instead of looking at the piecewise linear function, we follow \cite{F-W} to consider
$$q_{\Lambda, \beta}(z) := \frac\Lambda{\sqrt{\beta}} \tanh(\beta z).$$

With small $\beta>0$, this turns out to be a better approximation to the solution of the minimizer when the potential energy is low.

We have that
\begin{equation}\label{eq:kinetic_new}
		T(q_{\Lambda,\beta}) = \frac{1}{2} \Lambda^2\int_{\R} \sech^4(z)dz.
\end{equation}

Furthermore,
\begin{align*}
		\phi_{\Lambda,\beta}(z,\xi) & : = q_{\Lambda,\beta}(z+\frac{\xi}{2}) - q_{\Lambda,\beta}(z-\frac{\xi}{2}) = \frac\Lambda{\sqrt{\beta}} \bigg[\tanh(\beta(z+\frac{\xi}{2})) - \tanh(\beta(z-\frac{\xi}{2})) \bigg] \\
		&= \frac\Lambda{\sqrt{\beta}} \frac{2\tanh(\frac{\beta\xi}{2})\sech^2(\beta z)}{1-\tanh^2(\beta z)\tanh^2(\frac{\beta\xi}{2})} \\
		&= \frac\Lambda{\sqrt{\beta}} \bigg[\beta \xi - \frac{(\beta\xi)^3}{12} + O(\beta^5\xi^5)\bigg] \bigg[1 + \frac{1}{4}\beta^2\xi^2\tanh^2(\beta z) + O(\beta^4\xi^4) \bigg]\sech^2(\beta z) \\
		&= \frac\Lambda{\sqrt{\beta}} \bigg[\beta\xi + \frac{1}{4} (\beta\xi)^3\tanh^2(\beta z) - \frac{(\beta\xi)^3}{12} + O(\beta^5\xi^5) \bigg]\sech^2(\beta z).
\end{align*}

By a change of variable, we have that
\begin{align}
\nonumber
		\cE^{\ell}&(q_{\Lambda,\beta}) = \int_{\R} \int_{\ell}^\delta \bigg[\frac{1}{2}\bigg(\frac{\phi_{\Lambda,\beta}(z,\xi)}{\xi} \bigg)^2 + \frac{1}{6}\bigg(\frac{\phi_{\Lambda,\beta}(z,\xi)}{\xi} \bigg)^3 \bigg]\xi \,d\xi \,dz \\
\nonumber
		=&  \frac{\Lambda^2}{2\beta}\int_{\R}\int_{\ell}^\delta \bigg[ \beta\xi + \frac{1}{4}(\beta\xi)^3 \tanh^2(\beta z) - \frac{(\beta\xi)^3}{12} + O(\beta^5\xi^5) \bigg]^2\frac{\sech^4(\beta z)}{\xi} \,d\xi \,dz \\
\nonumber
		& + \frac{\Lambda^3}{6\beta^{3/2}}\int_{\R}\int_{\ell}^\delta \bigg[ \beta\xi + \frac{1}{4}(\beta\xi)^3 \tanh^2(\beta z) - \frac{(\beta\xi)^3}{12} + O(\beta^5\xi^5) \bigg]^3\frac{\sech^6(\beta z)}{\xi^2} \,d\xi \,dz \\
\nonumber
		= & \frac{\Lambda^2}{2\beta}\int_{\R}\int_{\ell}^\delta \bigg[ \beta \xi + 2(\beta \xi)^3(\frac14\tanh^2(z)-\frac{1}{12}) + O(\beta^5\xi^5) \bigg]\sech^4(z)\,d\xi \,dz \\
\nonumber
		&+ \frac{\Lambda^3}{6\beta^{3/2}} \int_{\R} \int_{\ell}^\delta \bigg[\beta^2 \xi + 3\beta^4\xi^3(\frac{1}{4}\tanh^2(z) -\frac{1}{12}) + O(\beta^8 \xi^7) \bigg]\sech^6(z) \,d\xi \,dz \\
\nonumber
		= & \frac{\Lambda^2}{2}\int_{\R} \bigg[\frac{1}{2}(\delta^2-\ell^2) + \frac{\beta^2}{4}(\delta^4 -\ell^4)(\frac{1}{2}\tanh^2(z)-\frac{1}{6}) + O(\delta^6\beta^4) \bigg]\sech^4(z) \,dz \\
\nonumber
		&+ \frac{\Lambda^3}{6} \int_{\R} \bigg[\frac{1}{2}(\delta^2-\ell^2)\beta^{1/2} + \frac{3}{4}\beta^{5/2} (\delta^4-\ell^4)(\frac{1}{4}\tanh^2(z)-\frac{1}{12}) + O(\delta^8\beta^{13/2}) \bigg]\sech^6(z)\,dz  \\
		&= C_1\frac{\Lambda^2}{4} (\delta^2 - \ell^2) + C_2 \frac{\Lambda^2\beta^2}{8}(\delta^4 - \ell^4)
		+ C_3\frac{\Lambda^3\beta^{1/2}}{12}(\delta^2-\ell^2) + O(\Lambda^2\beta^4+\Lambda^3\beta^{5/2}),  \label{e:asymptotic}
\end{align}
		where $C_1, C_3 >0$ and $C_2 <0$.
		 We choose $\Lambda_\ell(\beta)$ so that
		$$\cE^{\ell}(q_{\Lambda_{\ell}(\beta),\beta})=K.$$

		 Note that from \eqref{e:asymptotic}, we have
		$$\lim_{\beta\to 0} \Lambda_{\ell}(\beta) = \sqrt{\frac{4K}{C_1(\delta^2-\ell^2)}}=: M_\ell.$$

		Thus, we can pick $\beta$ small enough so that
by \eqref{eq:kinetic_new},
		$$K =\frac{\delta^2 -\ell^2}{2} T(q_{\Lambda_\ell(\beta),\beta})
+ C_3\frac{M_\ell^3 \beta^{1/2}}{12}(\delta^2-\ell^2) + o(\beta^{1/2})$$

		So, for any $K>0$, we can find a fixed $\beta>0$ so small that
		$$ K \ge \frac{\delta^2-\ell^2}{2}T(q_{\Lambda_{\ell},\beta}) +C(\beta),$$
		where $C(\beta)>0$ independent of $\ell$. This proves the inequality in Lemma \ref{lem:energy_revisit} for low energy, which is what we need to prove both existence of the minimizer to the approximate problem and to be able to pass to the limit using the similar argument in the previous section.

		\begin{remark}
				Taking $\beta$ small in the expression $\frac{\Lambda}{\sqrt{\beta}}\tanh(\beta z)$ here is similar to taking the length $L$ large in the piecewise linear approximating function in sections~\ref{properties} and \ref{limiting}. Recall that in Lemma~\ref{lem:energy_revisit} we needed to make $L$ large enough to create a constant $C$ independent of $\ell$ because the mass in the middle of the function is independent of $\ell$.
		\end{remark}

		\section{Discussion}\label{discussion}
		While with the $\tanh$ function in section \ref{low-energy}, one could establish existence of travelling waves in low potential energy setting, it is not clear that this argument would work for more general potentials, even when one could have existence for high potential energy. This is hinted in the work of \cite{F-W}---one needs to be more clever to obtain such result with more general potentials.

		An outstanding question is that about the compactness of the derivatives of the travelling waves. We answer this question by the following short argument.
		\\
		\begin{proposition}
				Let $q$ be a $C^2$ and increasing (resp. decreasing) function that solves \eqref{eq:limiting}. Suppose that $f(x,\xi)$ has a sign for $x\in (0,\infty)$ (resp. $x\in (-\infty,0)$). Then $\supp\{q'\}$ is unbounded above and below.
		\end{proposition}
		\begin{proof}
				We consider only the case $q$ is increasing and $f(x,\xi)>0$ for $x\in (0,\infty)$, $\xi\in(0,\delta)$. The other cases are similar.

				We argue by contradiction. Suppose that $q'$ has compact support. Let $z^+= \sup\{ \supp\{q'\}\}$. Consider $z = z^+ +\delta/2$. By \eqref{eq:limiting} and the assumption, we have
$q(z)<q(z^+)$ for $z<z^+$, hence
				\begin{align*}
						0 = c^2 q''(z) &= \int_0^\delta \bigg(f(q(z+\xi)-q(z),\xi) - f(q(z)-q(z-\xi),\xi)\bigg)  \, d\xi\\
							&= -\int_0^\delta f(q(z)-q(z-\xi),\xi)\, d\xi <0,
				\end{align*}
				which is a contradiction.
		\end{proof}

		While we know that the support of $q'$ is not compact, because of numerical evidence in \cite{S} (that the solutions with compact supports can approximate well the analytical solutions)  and the result of Friesecke and Pego for discrete lattices in \cite{F-P}, we conjecture that the solutions to \eqref{eq:limiting} decay exponentially to 0.

        Finally, we note that our approach is only one possible approach to prove existence of solitary waves in one dimensional peridynamics. Since there are quite a few results about the existence of solitary waves for Fermi-Pasta-Ulam-Tsingou's system, it would be worthwhile trying to adapt other approaches to peridynamics. We mention, in particular, the approach in \cite{herrmann_2010}, in which the author solves a constrained maximization problem and argues that certain maximizers are fixed points of a so-called {\it improvement operator} $\mathcal{T}$, which in turns solve the travelling wave equation. This approach is appealing since it provides good numerical results for Fermi-Pasta-Ulam-Tsingou's system and has a rather clean and interesting abstract framework. 

\section*{Acknowledgments}
		T.-S.V. would like to thank Gautam Iyer and Hung V. Tran for encouragement
		and useful comments.
		This material is based upon work partially supported by
		the National Science Foundation under grants
		DMS-1515400 and DMS-1812609, the Simons Foundation under grant \#395796,
		and the Center for Nonlinear Analysis, through NSF grant OISE-0967140.

\bibliography{bibliography}
\bibliographystyle{plain}
\end{document}